\documentclass[11pt]{article}
\usepackage[centertags]{amsmath}
\usepackage{amsfonts}
\usepackage{amssymb}
\usepackage{amsthm,color}
\usepackage{newlfont}
\usepackage{bbm}
\usepackage{graphicx}

\newtheorem{Theorem}{Theorem}[section]
\newtheorem{Definition}[Theorem]{Definition}
\newtheorem{Proposition}[Theorem]{Proposition}
\newtheorem{Lemma}[Theorem]{Lemma}

\newtheorem{Remark}[Theorem]{Remark}

\newtheorem{Assumption}{Assumption}

\numberwithin{equation}{section}

\begin{document}
\renewcommand{\figurename}{Fig.1}

\def\r2{\mathbb{R}^2}
\def\le{\left}
\def\r{\right}
\def\cost{\mbox{const}}
\def\a{\alpha}
\def\d{\delta}
\def\ph{\varphi}
\def\e{\epsilon}
\def\la{\lambda}
\def\si{\sigma}
\def\La{\Lambda}
\def\B{{\cal B}}
\def\A{{\mathcal A}}
\def\L{{\mathcal L}}
\def\O{{\mathcal O}}
\def\bO{\overline{{\mathcal O}}}
\def\F{{\mathcal F}}
\def\K{{\mathcal K}}
\def\H{{\mathcal H}}
\def\D{{\mathcal D}}
\def\C{{\mathcal C}}
\def\M{{\mathcal M}}
\def\N{{\mathcal N}}
\def\G{{\mathcal G}}
\def\T{{\mathcal T}}
\def\R{{\mathbb R}}
\def\I{{\mathcal I}}

\def\bw{\overline{W}}
\def\phin{\|\varphi\|_{0}}
\def\s0t{\sup_{t \in [0,T]}}
\def\lt{\lim_{t\rightarrow 0}}
\def\iot{\int_{0}^{t}}
\def\ioi{\int_0^{+\infty}}
\def\ds{\displaystyle}
\def\pag{\vfill\eject}
\def\fine{\par\vfill\supereject\end}
\def\acapo{\hfill\break}

\def\beq{\begin{equation}}
\def\eeq{\end{equation}}
\def\barr{\begin{array}}
\def\earr{\end{array}}
\def\vs{\vspace{.01mm}   \\}
\def\rd{\reals\,^{d}}
\def\rn{\reals\,^{n}}
\def\rr{\reals\,^{r}}
\def\bD{\overline{{\mathcal D}}}
\newcommand{\dimo}{\hfill \break {\bf Proof - }}
\newcommand{\nat}{\mathbb N}
\newcommand{\E}{\mathbb E}
\newcommand{\Pro}{\mathbb P}
\newcommand{\com}{{\scriptstyle \circ}}
\newcommand{\reals}{\mathbb R}

\def\Amu{{A_\mu}}
\def\Qmu{{Q_\mu}}
\def\Smu{{S_\mu}}
\def\H{{\mathcal{H}}}
\def\Im{{\textnormal{Im }}}
\def\Tr{{\textnormal{Tr}}}
\def\E{{\mathbb{E}}}
\def\P{{\mathbb{P}}}
\def\spn{{\textnormal{span}}}
\title{A Smoluchowski-Kramers approximation for an infinite dimensional system with state-dependent damping
}
\author{Sandra Cerrai\thanks{Department of Mathematics, University of Maryland, College Park, MD
20742, USA. Emails: cerrai@umd.edu, gxi@umd.edu}\ \,\thanks{Partially supported by NSF grants DMS-1712934 -  {\em Analysis of Stochastic Partial Differential Equations with Multiple Scales}  and  DMS-1954299 - {\em 
 Multiscale Analysis of Infinite-Dimensional Stochastic Systems}. } \ and Guangyu Xi\footnotemark[2]\\
\vspace{.1cm}\\
Department of Mathematics\\
 University of Maryland\\
College Park, Maryland, USA
}

\date{}

\maketitle

\begin{abstract}
We study the validity of a Smoluchowski-Kramers approximation for a  class of  wave equations in a bounded domain of $\mathbb{R}^n$ subject to a state-dependent damping and perturbed by a multiplicative noise. We prove that in the small mass limit the solution converges to the solution of a stochastic quasilinear  parabolic equation where a noise-induced extra drift is created.

\vspace{.3cm}

\end{abstract}

\section{Introduction}
In this article we study the following class of stochastic wave equations with state-dependent damping

\begin{equation}
\label{SPDE1}
\le\{\begin{array}{l}
\ds{\mu\partial_t^2 u_\mu=\Delta u_\mu - \gamma (u_\mu) \partial_t u_\mu + f(u_\mu)+ \si(u_\mu)\partial_t w^Q,\ \  t>0,\ \ \ \ x\in \mathcal{O},}\\[10pt]
\ds{u_\mu(0)=u_0,\ \ \ \ \partial_t u_\mu(0)=v_0,\ \ \ \ \ \ \  u_{\mu_{|_{\partial \mathcal{O}}}}=0,}
\end{array}\r.
\end{equation}
and their small mass limit as $\mu \rightarrow 0$. Here $\mathcal{O}$ is a bounded domain on $\mathbb{R}^n$, and $w^Q(t,x)$ is a cylindrical Wiener process, which is white in time and colored in space. The friction coefficient $\gamma$ is a strictly positive, bounded and continuously differentiable function, and $f$ and $\si$  are Lipschitz continuous functions.

By Newton's second law of motion, the solution $u_\mu(t,x)$ of equation \eqref{SPDE1} can be interpreted as the displacement field of the particles in a continuum body occupying domain $\mathcal{O}$, subject to a random external force field $\sigma(u_\mu)\partial_t w^Q$ and a state-dependent damping force $\gamma (u_\mu) \partial_t u_\mu$, which is proportional to the velocity field. In addition, the particles are subject to the interaction forces between neighboring particles represented by the Laplace operator $\Delta$, and the non-linear reaction represented by $f$. Here  $\mu$ represents the constant density of the particles and we are interested in the regime when $\mu\to 0$, which is the so-called Smoluchowski-Kramers approximation limit (ref. \cite{kra} and \cite{smolu}).
 
A series of papers (ref. \cite{CF1}, \cite{CF2} and \cite{salins}) studied the limiting behavior of $u_\mu$, for a large class of reaction terms $f$, and for both additive  and multiplicative noise. In all those papers the friction coefficient $\gamma$ is assumed to be constant and  a perturbative limit is obtained. Namely, it is  proved that in the small mass limit $u_\mu$ converges to the solution of the following parabolic problem
 \begin{equation}
\label{eq1-par}
\le\{\begin{array}{l}
\ds{\gamma\,\partial_t u=\Delta u+f(u)+\si(u)\partial_tw^Q,}\\
\vs
\ds{u(0)=u_0,\ \ \ \ \ \ u_{|_{\partial \mathcal{O}}}=0.}
\end{array}
\r.
\end{equation}
More precisely, it is  shown that for every $T>0$ and $\eta>0$
\begin{equation}
\label{intro2}
\lim_{\mu\to 0}\mathbb{P}\le(\sup_{t \in\,[0,T]}\Vert u_\mu(t)-u(t)\Vert_{L^2(\mathcal{O})}>\eta\r)=0.
\end{equation}
In fact, in \cite{CS3} it is proved that when $f$ is Lipschitz continuous, the following stronger convergence holds
\begin{equation}
\label{intro1}
\lim_{\mu\to 0}\mathbb{E} \sup_{t \in\,[0,T]}\Vert u_\mu(t)-u(t)\Vert_{L^2(\mathcal{O})}^p=0,
\end{equation}
for every $p\geq 1$.  Note that several problems related to this type of limit have been addressed in a variety of finite and infinite dimensional contexts (see e.g. \cite{f} and \cite{spi}, for the finite dimensional case, and \cite{CF1},  \cite{CF2},  \cite{CFS},  \cite{Lv2}, \cite{Lv3}, \cite{Lv4}, \cite{Nguyen} and \cite{salins}, for the infinite dimensional case).

Once proved the validity of the small mass limit in any fixed time interval, it is important to understand how stable this limit is for long times. To this purpose, \cite{CGH} studies the convergence of the statistically invariant states for a class of semi-linear wave equations with linear damping, i.e. equation \eqref{SPDE1} 
with constant friction coefficient $\gamma$, both with Lipschitz and with polynomial non-linearity $f$. A similar problem is studied in \cite{CF1} when the two systems are of gradient type. In that case the Boltzmann distribution for the solution of the second order equation is explicitly given in terms of a Gibbs measure. It turns out that the first marginal of the Boltzmann distribution does not depend on $\mu$ and coincides with the invariant measure of the limiting first order equation. In the case studied in \cite{CGH}, there is no explicit expression  for the invariant distributions of \eqref{SPDE1}. Nevertheless, it is shown  that the first marginals of any sequence of invariant measures for \eqref{SPDE1} converge in a suitable Wasserstein metric to the unique invariant measure of equation \eqref{eq1-par}.
In the same spirit, \cite{sal} and \cite{sal2} studied the convergence of the quasi-potentials $V_\mu(u,v)$, which describe the asymptotics of the exit times and the large deviation principle for the invariant measures to equation \eqref{SPDE1}. In \cite{sal}, gradient systems are considered, so that $V_\mu$ is explicitly computed and it is shown that $V^\mu(u)$, the infimum of $V_\mu(u,v)$ over all $v \in\,H^{-1}(\mathcal{O})$, coincides  for every $\mu>0$ with $V(u)$, the quasi-potential associated with equation \eqref{eq1-par}. In \cite{sal2}, the non-gradient case is studied and it is shown that $V^\mu(u)$ converges pointwise to $V(u)$, as $\mu$ goes to zero.

In all the aforementioned papers, the case of a constant friction coefficient $\gamma$ is considered and the limiting equation \eqref{eq1-par} is formally obtained  by taking $\mu=0$ in \eqref{SPDE1}. However, there are relevant situations in which this is not true. This happens, for example, in the case when the constant friction is replaced by a magnetic field. As a matter of fact, even in the case of a constant magnetic field and finite dimension,  the small mass limit does not yield the solution of the first order equation (ref. 
\cite{CF3}, \cite{CWZ} and \cite{lee} for the finite dimensional case and \cite{CS3} for the infinite dimensional case).  In this case, a possible strategy consists in regularizing  the problem  by adding a small friction or by smoothing the noise in time and,  in the double limit, it is possible to give a meaning to the  Smoluchowski-Kramers approximation. Notice that in  \cite{CS3} the limiting equation is an SPDE of hyperbolic type.

In the present paper, we are dealing with another situation when  the small mass limit does not give a perturbative result. As mentioned at the beginning of this introduction, we consider a wave equation perturbed by a multiplicative noise,  having a friction term whose intensity is state-dependent. This problem has been extensively studied in finite dimension in a series of papers (see \cite{hmdvw}, and references therein, and also \cite{hhv}). In these papers it is shown how the interplay between the non-constant friction coefficient and the noise creates an additional drift in the limiting first order equation, when the mass $\mu$ goes to zero. More precisely, the following system is studied
\[\le\{ \begin{array}{l}
\ds{d x_\mu(t) = v_\mu(t)\,dt,\ \ \ \ \ x_\mu(0)=x \in\,\mathbb{R}^d,}\\[10pt]
\ds{\mu dv_\mu(t)=\le[b(x_\mu(t))-\gamma(x_\mu(t))v_\mu(t)\r]\,dt+\si(x_\mu(t))\,dW(t),\ \ \ \ \ v_\mu(0)=v \in\,\mathbb{R}^d,}
\end{array}\r.\]
where $\gamma$ is a matrix valued function defined on $\mathbb{R}^d$, such that for some positive $\gamma_0$
\[\inf_{x \in\,\mathcal{O}}\,\xi^T \gamma(x) \xi\geq \gamma_0 \vert\xi\vert^2,\ \ \ \ \xi \in\,\mathbb{R}^d.\]
It is proved that, as $\mu$ goes to zero, $x_\mu$ converges in $L^2$, with  respect to the uniform norm in $C([0,T];\mathbb{R}^d)$, to the solution of the first order equation
\[dx(t)=\le(\frac{b(x(t))}{\gamma(x(t))} +S(x(t))\r) dt+\frac{\si(x(t))}{\gamma(x(t))}\,dW(t),\ \ \ \ \ x(0)=x,\]
where the noise induced drift $S(x)$ is given by
\[S_i(x)=\frac{\partial}{\partial x_l}\le[(\gamma^{-1})_{ij}(x)\r]J_{jl}(x),\]
and the matrix valued function $J$ is the solution of the Lyapunov equation
\[J(x)\gamma^\star(x)+\gamma(x)J(x)=\si(x)\si^\star(x).\]

Our purpose here is to understand if something similar happens also in the case of infinite dimensional systems. In fact, in what follows we will prove that for every initial condition $(u_0, v_0) \in\,H^1(\mathcal{O})\times L^2(\mathcal{O})$ and for every $\d>0$ and $p<\infty,$ 
\begin{equation}
\label{limitefinale}	
\lim_{\mu\to 0} \mathbb{P}\le(\Vert u_\mu-u\Vert_{C([0,T];H^{-\d}(\mathcal{O}))}+ \Vert u_\mu-u\Vert_{L^p(\mathcal{O})}>\eta\r)=0,\ \ \ \ \eta>0, \end{equation}
where $u$ is the unique solution of the quasilinear stochastic parabolic equation 
\begin{equation}
\label{gs40}
\le\{\begin{array}{l}
\ds{\partial_t u= \frac{1}{\gamma(u)}\Delta u +\frac{f(u)}{\gamma(u)} -\frac{\gamma'(u)}{2\gamma^3(u)} \sum_{i=1}^\infty (\si(u)Qe_i)^2 +\frac{\si(u)}{\gamma(u)}\partial_t w^Q,\ \  t>0,\ \ \ \ x\in \mathcal{O},}\\[18pt]
\ds{u(0)=u_0, \ \ \ \ \ \ \  u_{|_{\partial\mathcal{O}}}=0.}
\end{array}\r.
\end{equation}
It is important to notice that if $\si$ is constant, then the noise induced term 
\[H(u):=-\frac{\gamma'(u)}{2\gamma^3(u)} \sum_{i=1}^\infty (\si(u)Qe_i)^2,\]
coincides with the Stratonovich-to-It\^o correction, so that equation \eqref{gs40} can be written as
\[\partial_t u= \frac{1}{\gamma(u)}\Delta u +\frac{f(u)}{\gamma(u)} +\frac{\si}{\gamma(u)}\circ \partial_t w^Q.\]
However, if  $G(u)$ denotes the It\^o-to-Stratonovich correction
\[G(u)=-\frac 12 \sum_{i=1}^\infty \partial_u\left(\frac{\si(u)Qe_i}{\gamma(u)}\right)\left(\frac{\si(u)Qe_i}{\gamma(u)}\right),\]
then
\[H(u)+G(u)=-\frac 1{2\gamma^2(u)}\,\sum_{i=1}^\infty
\left(\si(u)Qe_i\right)\, \partial_u\left(\si(u)Qe_i\right),\]
and this is manifestly non-trivial in general, when $\sigma$ is not constant.
 This means that in the case of an arbitrary state-dependent diffusion coefficient $\sigma$ the small mass limit does not lead to the perturbative parabolic quasilinear equation, obtained by taking $\mu=0$ and replacing the It\^o's  with the   Stratonovich's integral.

We would also like to point out that, unlike in finite dimension, here we are not handling systems of equations. This means in particular that $\gamma$ is  a scalar function and for every function $u:[0,T]\times \mathcal{O}\to\mathbb{R}$ we can write
\begin{equation}
\label{gammag}	
\gamma(u(t,x))\partial_t u(t,x)=\partial_t [g(u(t,x))],\ \ \ \ t \in\,[0,T],\ \ \ x \in\,\mathcal{O},\end{equation}
where $g^\prime=\gamma$. The case of systems and of matrix valued friction coefficients requires a different analysis and will be investigated
in a forthcoming paper.

\medskip

Our first step in the proof of  \eqref{limitefinale} is proving that, for every fixed $\mu>0$, equation \eqref{SPDE1} has a unique adapted solution $(u_\mu,\partial_t u_\mu) \in\,L^2(\Omega;C([0,T];H^1(\mathcal{O})\times L^2(\mathcal{O})))$. Because of the the non-constant friction, it is conveniente to reformulate equation \eqref{SPDE1} in terms of the new variables $(u_\mu,g(u_\mu)/\mu+\partial_t u_\mu)$. However, due to the presence of the non-linear term $g(u)$,  using   the theory of linear semigroups, as done in the previous papers \cite{CF1}, \cite{CF2} and \cite{CS3}, turns out to be the wrong path to follow. Instead, here it is more appropriate to use the theory of monotone non-linear operators (see \cite{Barbu}). 

Once proved the well-posedness of \eqref{SPDE1}, next we prove the uniform bounds of the solutions $(u_\mu,\partial_t u_\mu)$, which are required to obtain tightness.   This is one of the most delicate parts of the paper. Actually, even when using the It\^o formula for a nicely chosen energy functional, the more classical arguments that work in finite dimension fail. Nevertheless, we are able to prove that $u_\mu$ is bounded with respect to $ \mu $ in $ L^2(\Omega;C([0,T];L^2(\mathcal{O}))\cap L^2(0,T;H^1(\mathcal{O}))) $. Even more delicate are the bounds for the velocity $\partial_t u_\mu$. Of course, we know that we cannot   have any uniform bounds with respect to $\mu$. However,  we expected to have 
\[\sup_{\mu \in\,(0,1)}\mu^\a\,\mathbb{E}\sup_{t \in\,[0,T]}\,\Vert \partial_t u_\mu(t)\Vert^2_{L^2(\mathcal{O})}<\infty,\]
for  $\a=1$.
As a matter of fact, by using an argument by contradiction we can prove the bound above only for $\a= 3/2$, but this is enough to obtain the fundamental limit
\begin{equation}
\label{introlim}
\lim_{\mu\to 0}\mu\,\mathbb{E}\sup_{t \in\,[0,T]}\,\Vert \partial_t u_\mu(t)\Vert_{L^2(\mathcal{O})}=0.\end{equation}

After defining $\rho_\mu=g(u_\mu)$, these uniform bounds are fundamental to prove the tightness of  the family $(\rho_\mu)_{\mu>0}$ in $L^p(0,T;L^2(\mathcal{O}))\cap C([0,T];H^{-\d}(\mathcal{O}))$, for every $p<\infty$ and  $\d>0$. We show that for every $\mu>0$ the function $\rho_\mu$ solves the equation
\[ \begin{aligned}
 	\rho_\mu(t)+\mu \partial_t u_\mu(t) & =g(u_0)+\mu v_0 +\int_0^t \text{div}[b(\rho_\mu(s))\nabla \rho_\mu(s)] ds\\[10pt]
 	&\quad + \int_0^t F(\rho_\mu(s)) ds+ \int_0^t \si_g(\rho_\mu(s))dw^Q(s),
\end{aligned}
\]
where $b=1/\gamma \circ g^{-1}$, $F=f\circ g^{-1}$, and $\si_g(h)=\si(g^{-1}\circ h)$.  Working with this equation, instead of \eqref{SPDE1}, makes the use of the a-priori bounds and of limit \eqref{introlim} more direct. 

Once tightness is proved, we have the weak convergence of  the sequence $(\rho_\mu)_{\mu>0}$ to some $\rho$ that solves the quasilinear parabolic SPDE
\begin{equation}
\label{SPDERho-intro}
\le\{\begin{array}{l}
\ds{\partial_t \rho=\text{div}[b(\rho) \nabla \rho]+F(\rho)+\si_g(\rho)dw^Q(t), \ \ \ \  t>0, \ \ \ \ x\in \mathcal{O},}\\[10pt]
\ds{\rho(0,x)=g(u_0),\ \ \ \ \ \ \ \ \rho(t,x)=0,\ \ \ x\in \partial \mathcal{O}.}
\end{array}\r.
\end{equation}
Then, since we can prove  pathwise uniqueness for equation \eqref{SPDERho-intro}, from weak convergence we  get the convergence in probability. Finally, a generalized It\^o formula stated in the appendix allows us to get the convergence of $u_\mu$ to the solution of equation \eqref{gs40}.

We would like to remind that equations like \eqref{SPDERho-intro} have attracted a lot of attention in recent years, and several papers have studied their well-posedness in $C([0,T];L^2(\mathcal{O}))\cap L^2(0,T;H^1(\mathcal{O}))$, in case of periodic boundary conditions, under considerably more general assumptions on the coefficients $b$, that can be matrix valued and even degenerate (see \cite{DHV2016} and \cite{HZ2017}). Our $b$ here is scalar valued and non-degenerate, but this allows us to have, at least in the additive case, weaker  assumptions  on the regularity of the noise  than in \cite{DHV2016} and \cite{HZ2017}. In particular, as a byproduct of our small mass limit, we get the well-posedness of equation \eqref{SPDERho-intro} for a noise  that, for example in the case of constant $\si$ is only assumed to live in $L^2(\mathcal{O})$, which seems to be a new result.
\medskip

{\em Organization of the paper:} In Section \ref{sec2}, we introduce the notations and we describe the assumptions we make on the coefficients and on the noise in equation \eqref{SPDE1}. In section \ref{sec3}, we study the well-posedness of equation \eqref{SPDE1} in space $L^2(\Omega;C([0,T];H^1(\mathcal{O})\times L^2(\mathcal{O})))$, for every $T>0$ and every fixed $\mu>0$. In Section \ref{sec5}, we prove some uniform bounds with respect to $\mu$
for the solutions $((u_\mu,\partial_t u_\mu))_{\mu>0}$ in 	suitable functional spaces. In Section \ref{sec6}, these bounds allow us to prove the tightness of the distributions of $g(u_\mu)$, for $\mu>0$ sufficiently small. In Section \ref{sec4}, we prove the validity of pathwise uniqueness for equation \eqref{SPDERho-intro}. In Section \ref{sec7}, we give the proof of the convergence in probability of $u_\mu$, as $\mu$ goes to zero and we identify the limit $u$ as the solution of the first order equation \eqref{gs40}.

\section{Preliminaries}
\label{sec2}

Throughout the present paper $\mathcal{O}$ is a bounded domain in $\mathbb{R}^n$, with $n\geq 1$, and it has a boundary of class $C^3$. We denote by $H$ the Hilbert space $L^2(\mathcal{O})$ and by $\langle \cdot , \cdot \rangle_H$ the corresponding  inner product.  $H^1$ is the completion of $C_0^\infty(\mathcal{O})$ with respect to norm 
\[
\Vert u \Vert_{H^1}^2:=\Vert \nabla u \Vert_H^2=\int_\mathcal{O} \vert \nabla u(x)\vert^2 dx,
\]
and $H^{-1}$ is the dual space to $H^1$. Then $H^1$, $H$ and $H^{-1}$ are all complete separable metric spaces, and  the following relation between them holds
\begin{equation}
    H^1\subset H\subset H^{-1},
\end{equation}
where both inclusions  are compact embeddings. In what follows,  we shall denote  
\begin{equation*}
    \mathcal{H}=H\times H^{-1},\ \ \ \ \  \mathcal{H}_1=H^1 \times H. 
\end{equation*}

Given the domain $\mathcal{O}$, we denote by $(e_i)_{i\in\mathbb{N}}\subset H^1$ the complete orthonormal basis of $H$ which diagonalizes the Laplacian $\Delta$, endowed with Dirichlet boundary conditions on $\partial\mathcal{O}$. Moreover, we denote by $(-\alpha_i)_{i\in \mathbb{N}}$ the corresponding sequence of eigenvalues, i.e.
 \[\Delta e_i=-\alpha_i e_i,\ \ \ \  i\in \mathbb{N}.\]
  Given 
  \[u=\sum_{i=1}^\infty b_i e_i,\ \ \ \  v=\sum_{i=1}^\infty c_i e_i,\]
  for some sequences of real numbers $(b_i)_{i\in\,\mathbb{N}}$ and $(c_i)_{i\in\,\mathbb{N}}$,
we have
\begin{equation}\label{InnerProducts}
    \langle u,v \rangle_{H^1}=\sum_{i=1}^\infty \alpha_i b_i c_i,\quad \langle u,v \rangle_{H}=\sum_{i=1}^\infty b_i c_i,\quad \langle u,v\rangle_{H^{-1}}=\sum_{i=1}^\infty \frac{1}{\alpha_i} b_i c_i.
\end{equation}
From \eqref{InnerProducts} we can derive the Poincar\'e inequality
\begin{equation}\label{Poincare}
    \Vert u\Vert _{H}\leq \frac{1}{\sqrt{\alpha_1}}\Vert u\Vert _{H^1},\ \ \ \ u\in H^1,\ \ \ \ \  \Vert u\Vert _{H^{-1}}\leq \frac{1}{\sqrt{\alpha_1}} \Vert u\Vert _H,\ \ \ \ u\in H.
\end{equation}

As for the stochastic perturbation, we assume that $w^Q(t)$ is a cylindrical $Q$-Wiener process, defined on a complete stochastic basis  $(\Omega,\mathcal{F},(\mathcal{F}_t)_{t\geq 0},\mathbb{P})$. This means that $w^Q(t)$ can be formally written as 
\[
w^Q(t)=\sum_{i=1}^\infty Q e_i \beta_i(t),
\]
where $(\beta_i)_{i\in \mathbb{N}}$ is a sequence of independent standard Brownian motions on $(\Omega,\mathcal{F},(\mathcal{F}_t)_{t\geq 0},\mathbb{P})$, $Q:H\rightarrow H$ is a bounded  linear operator, and $(e_i)_{i\in\,\mathbb{N}}$ is the complete orthonormal system introduced above that diagonalizes the Laplace operator, endowed with Dirichlet boundary conditions. 

In what follows we shall denote by $H_Q$ the set $Q(H)$. $H_Q$ is the reproducing kernel of the noise $w^Q$ and  is a Hilbert space, endowed with the inner product
\[\langle Qh, Qk\rangle_{H_Q}=\langle h,k\rangle_H,\ \ \ \ h, k \in\,H.\] Notice that the sequence $(Q e_i)_{i \in\,\mathbb{N}}$ is a complete orthonormal system in $H_Q$. Moreover, if $U$ is any Hilbert space containing $H_Q$ such that the embedding  of $H_Q$ into $U$ is Hilbert-Schmidt, we have that 
\begin{equation}
  \label{contb}w^Q \in\,C([0,T];U).
\end{equation}

Next, we recall that for every two separable Hilbert spaces $E$ and $F$, $\mathcal{L}_2(E,F)$ denotes the space of Hilbert-Schmidt operators from $E$ into $F$. $\mathcal{L}_2(E,F)$ is a Hilbert space, endowed with the inner product
\[\langle A,B\rangle_{\mathcal{L}_2(E,F)}=\mbox{Tr}_E\,[A^\star B]=\mbox{Tr}_F[B A^\star].\]

Throughout this article, we will always assume that the three hypotheses below are true.
\begin{Assumption}
\label{Assumption1}
The mapping $\si:H\to \mathcal{L}_2(H_Q,H)$ is defined by \[[\sigma(h)Qe_i](x) = \sigma_i(x,h(x)), \ \ \ \ x \in\,\mathcal{O},\] for every $h\in H$ and $i\in\,\mathbb{N}$, for some mapping $\sigma_i:\mathcal{O}\times \mathbb{R}\rightarrow \mathbb{R}$. We assume $\sigma$ is bounded, that is
\begin{equation*}
    \si_\infty:=\sup_{h\in H} \Vert\si(h)\Vert_{\mathcal{L}_2(H_Q,H)}<\infty,
\end{equation*}
and 
\begin{equation}
\label{sgfine1}	
\sup_{x \in\,\mathcal{O}}\, \sum_{i=1}^\infty \vert \sigma_i(x,y_1) - \sigma_i(x,y_2)\vert^2 \leq L\,\vert y_1-y_2\vert^2,\ \ \ \ \ y_1, y_2 \in\,\mathbb{R},
\end{equation}
\end{Assumption}

Notice that \eqref{sgfine1} implies $\si$ is Lipschitz continuous in the sense that
 for any $h_1,h_2\in H$
\begin{equation*}
    \Vert\si(h_1)-\si(h_2)\Vert_{\mathcal{L}_2(H_Q,H)}^2 \leq L\, \Vert h_1 -h_2\Vert_H^2.
\end{equation*}

\begin{Remark}
{\em If $\si$ is constant,  then Assumption \ref{Assumption1} means that $\si Q$ is a Hilbert-Schmidt operator in $H$. Equivalently, in case $\si$ is the identity operator, this means that the noise $w^Q$ lives in $H$, so that we can take $U=H$.

If $\si$ is not constant, then Assumption \ref{Assumption1} is satisfied if for example
\[ [\si(h)Qk](x)=\la(h(x))Qk(x),\ \ \ \ x \in\,\mathcal{O}, \ \ \ \ h, k \in\,H,\]
for some $\la:\mathbb{R}\to \mathbb{R}$ bounded and Lipschitz continuous and for some $Q \in\,\mathcal{L}(H)$ such that
\[\sum_{i=1}^\infty \Vert Qe_i\Vert^2_{L^\infty(\mathcal{O})}<\infty.\]
In case $Q$ is diagonalizable with respect the basis $(e_i)_{i \in\,\mathbb{N}}$, with $Q e_i=\la_i e_i$,
the condition above reads
\begin{equation}\label{gx005}
    \sum_{i=1}^\infty \la_i^2\Vert e_i\Vert^2_{L^\infty(\mathcal{O})}<\infty.
\end{equation}
In general (see \cite{greiser}), it holds that 
\[\Vert e_i\Vert_{L^\infty(\mathcal{O})} \leq c\, i^\alpha\]
for some $\alpha>0$, and \eqref{gx005} becomes
\[\sum_{i=1}^\infty \la_i^2 \, i^{2\alpha}<\infty.\]
In particular, when $n=1$ or the domain is a hyperrectangle in higher dimension, the eigenfunctions $(e_i)_{i\in\mathbb{N}}$ are equi-bounded and \eqref{gx005} becomes $\sum_{i=1}^\infty \la_i^2 \, <\infty$.
}
\end{Remark}

\begin{Assumption}\label{Assumption2}
The mapping $\gamma$ belongs to $C^1_b(\mathbb{R})$ and there exist $\gamma_0$ and $\gamma_1$ such that
\begin{equation}
\label{nonlinearity assumption}
0<\gamma_0\leq \gamma(r)\leq \gamma_1,\ \ \ \ \ \ r\in\mathbb{R}.
\end{equation}
\end{Assumption}

In what follows, we shall define
\[g(r)=\int_0^r \gamma(\si)\,d\si,\ \ \ \ \ r \in\,\mathbb{R}.\]
Clearly $g(0)=0$ and $g^\prime(r)=\gamma(r)$. In particular, due to  \eqref{nonlinearity assumption},  $g$ is uniformly Lipschitz continuous on $\mathbb{R}$. Moreover, $g$ is strictly increasing and
\begin{equation}
\label{gs51}
\le(g(r_1)-g(r_2)\r)(r_1-r_2)\geq \gamma_0\,|r_1-r_2|^2,\ \ \ \ \ r_1, r_2 \in\,\mathbb{R}.
\end{equation}

 \begin{Assumption}
\label{Assumption3}
The mapping $f:\mathbb{R}\to\mathbb{R}$ is Lipschitz continuous. Moreover, there exist $\la<\a_1$, $\d<1$ and $c\geq 0$ such that 
\begin{equation}
\label{gs25}
f(r)r\leq \la\ r^2+c\le(1+|r|^{1+\d}\r),\ \ \ \ \ r \in\,\mathbb{R}.\end{equation}
\end{Assumption}

Any Lipschitz continuous function $f$ having sub-linear growth satisfies \eqref{gs25}. Condition \eqref{gs25} allows also linear growth for $f$, but in this case we need
\[\sup_{r_1,r_2 \in\,\mathbb{R}}\frac{f(r_1)-f(r_2)}{r_1-r_2}<\a_1.\]
In particular, \eqref{gs25} is satisfied if
\[\Vert f \Vert_{\tiny{\text{Lip}}}<\a_1.\]

\section{Well-posedness of equation \eqref{SPDE1} }
\label{sec3}

In this section we study the existence and uniqueness of solutions to the non-linear stochastic wave equations \eqref{SPDE1} with initial data $(u_0,v_0)\in\mathcal{H}_1$, for every fixed $\mu>0$.  
Notice that  the second order equations \eqref{SPDE1} can be written as the following system
\begin{equation}
\label{SPDE2}
\le\{\begin{array}{l}
\ds{du_\mu(t)=v_\mu(t) dt,\qquad u_\mu (0)=u_0,}\\[10pt]
\ds{dv_\mu(t)=\frac{1}{\mu}[\Delta u_\mu(t) - \gamma (u_\mu(t)) v_\mu(t)+f(u_\mu(t))]dt+ \frac{1}{\mu}\si(u_\mu(t))	dw^Q(t),\qquad v_\mu(0)=v_0,}\\[10pt]
\ds{u_\mu(t)_{\vert_{\partial \mathcal{O}}}=0,\ \ \ \ t>0.}
\end{array}\r.
\end{equation}
Now, if we  define 
\begin{equation}
\label{gs15}
    \eta:=\partial_t u +\frac {g(u)}\mu,
\end{equation}
and $z=(u,\eta)$, 
 system \eqref{SPDE2} can be rewritten as 
\begin{equation}
\label{SPDE3}
dz_\mu(t)=A_\mu(z_\mu(t))\,dt+\Sigma_\mu(z_\mu(t))\,dw^Q(t),\ \ \ \ z_\mu(0)=\left(u_0,v_0+\frac{g(u_0)}{\mu}\right),
\end{equation}
where we denoted
\[\Sigma_\mu(u,\eta)=\frac 1\mu(0,\si(u)),\ \ \ \ (u,\eta) \in\,\mathcal{H},\]
and
\[A_\mu(u,\eta)
=\left(\frac{-g(u)}{\mu}+\eta,\frac{1}{\mu}[\Delta u+f(u)]\right),\ \ \ \ (u,\eta) \in\,D(A_{\mu})=\mathcal{H}_1.\]
This means that  the adapted $\mathcal{H}_1$-valued process $z_\mu(t)=(u_\mu(t),\eta_\mu(t))$ is the unique solution of the equation
\begin{equation}
\label{SPDE5}
z_\mu(t)=(u_0,g(u_0)/\mu+v_0)+\int_0^t A_\mu(z_\mu(s))\,ds+\int_0^t \Sigma_\mu(z_\mu(s))dw^Q(s),
\end{equation}
if and only if the adapted $\mathcal{H}_1$-valued process $(u_\mu(t),v_\mu(t)):=(u_\mu(t),-g(u_\mu(t))/\mu+\eta_\mu(t))$ is the unique solution of the system
\begin{equation}
\label{gs52}
\le\{
\begin{array}{l}
\ds{u_\mu(t)=u_0+\int_0^t v_\mu(s)\,ds}\\[10pt]
\ds{\mu v_\mu(t)=\mu v_0+\int_0^t \le[\Delta u_\mu(s)-\gamma(u_\mu(s))v_\mu(s)+f(u_\mu(s))\r]\,ds+\int_0^t\si(u_\mu(s))dw^Q(s).}
\end{array}\r.\end{equation}
In this section, we are interested in the well posedness of equation \eqref{SPDE2} (and equivalently of \eqref{SPDE3}) and not on the dependence  of its solution on $\mu$.  Thus, without any loss of generality,  we will only consider the case when $\mu=1$ and, for simplicity of notation,  we will denote $A_1$ and $\Sigma_1$ by $A$ and $\Sigma$, respectively. 

We start our study of equation \eqref{SPDE3} by analyzing the non-linear operator $A$. To this purpose, it is immediate to check that
\begin{equation}
\label{gs11}
\Vert A(z)\Vert_{\mathcal{H}}\leq c\,\le(1+\Vert z\Vert _{\mathcal{H}_1}\r),\ \ \ \ \ z \in\,D(A),
\end{equation}
because $f$ and $g$ are both Lipschitz continuous. In next lemma, we prove that the nonlinear operator $A:D(A)\subset \mathcal{H}\to \mathcal{H}$ is {\em quasi-m-dissipative}. For all the details on  the definitions and the  basic results about maximal monotone nonlinear operators that we are using below, we refer to \cite[Chapters 2 and 3]{Barbu}. 

\begin{Lemma}\label{lem: Boperator}
Under Assumptions \ref{Assumption2} and \ref{Assumption3}, there exists $\kappa\geq 0$ such that for every $z_1, z_2 \in\,D(A)$
\begin{equation}
\label{gs1}
\langle A(z_1)-A(z_2),z_1-z_2\rangle_{\mathcal{H}}\leq \kappa \,\Vert z_1-z_2\Vert ^2_{\mathcal{H}}.
\end{equation}
Moreover, there exists $\la_0>0$ such that
\begin{equation}
\label{gs2}
\text{{\em Range}}(I-\la A)=\mathcal{H},\ \ \ \ \ \la \in\,(0,\la_0).
\end{equation}
\end{Lemma}

\begin{proof} For every $z_1=(u_1,\eta_1)$ and $z_2=(u_2,\eta_2)$ in $D(A)$ we have
\begin{align*}
    \langle A(z_1)-A(z_2),z_1-z_2\rangle_{\mathcal{H}} &=-\langle g(u_1)-g(u_2),u_1-u_2\rangle_H+\langle \eta_1-\eta_2,u_{1}-u_2\rangle_H\\[10pt]
    &\quad +\langle \Delta u_1-\Delta u_2,\eta_1-\eta_2\rangle_{H^{-1}}+\langle f(u_1)-f(u_2),\eta_1-\eta_2\rangle_{H^{-1}}\\[10pt]
    &\leq -\gamma_0 \Vert u_1-u_2\Vert _H^2+\Vert f(u_1)-f(u_2)\Vert _{H^{-1}}\Vert \eta_1-\eta_2\Vert _{H^{-1}}\\[10pt]
    &\leq -\gamma_0 \Vert u_1-u_2\Vert _H^2+c\le(\Vert u_1-u_2\Vert _H^2+\Vert \eta_1-\eta_2\Vert ^2_{H^{-1}}\r),
\end{align*}
where the first inequality follows from \eqref{gs51}, and the second inequality follows from the Lipschitz continuity of $f$ and the Poincar\'e inequality \eqref{Poincare}.
In particular, there exists some $\kappa\geq 0$ such that \eqref{gs1} holds.

Next, in order to prove \eqref{gs2}, we need to show that, if $\la$ is sufficiently small, then for every $h=(h_1,h_2) \in\,\mathcal{H}$, there exists 
$z=(u, \eta) \in\,\mathcal{H}_1$ such that 
\[z-\la A(z)=h,\] or, equivalently, there exists some $u \in\,H^1$ such that
\begin{equation}
\label{gs3}
u-\la^2 \Delta u=-\la g(u)+\la^2 f(u)+(h_1+\la h_2).\end{equation}
In particular, if we define 
\[\Gamma_\la(u)=(I-\la^2 \Delta)^{-1}\le[-\la g(u)+\la^2 f(u)+(h_1+\la h_2)\r],\]
we need to prove that there exists some $\la_0>0$ such that  $\Gamma_\la:H\to H$ 
is a contraction, for every $\la \in\,(0,\la_0)$. By
\[\Vert(I-\la^2 \Delta)^{-1}\Vert_{\mathcal{L}(H)}\leq 1,\ \ \ \ \ \la>0,\]
we have
\begin{align*}
    \le\Vert \Gamma_\la(u_1)-\Gamma_\la(u_2)\r\Vert _H &\leq c\,\la\,\le(\Vert (g(u_1)-g(u_2)\Vert _H+\la\,\Vert f(u_1)-f(u_2)\Vert _H\r)\\[10pt]
    &\leq c\,\la(1+\la)\,\Vert u_1-u_2\Vert _H,
\end{align*}
which implies that $\Gamma_\lambda$ is a contraction for small enough $\lambda$. 
\end{proof}

Now, if we define $\bar{\la}:=\la_0\wedge \kappa^{-1}$, due to Lemma \ref{lem: Boperator} we have that the operator
\[J_\la(z):=(I-\la A)^{-1}(z),\ \ \ z\in\,\mathcal{H},\ \ \ \ \la \in\,(0,\bar{\la}),\]
is well-defined, and  is Lipschitz continuous from $\mathcal{H}$ into $\mathcal{H}$, with Lipschitz constant $(1-\la \kappa)^{-1}$ (see \cite[Proposition 3.2]{Barbu}).
Thus, for every $\la \in\,(0,\bar{\la})$, we can introduce the {\em Yosida approximation} of $A$, defined as
\begin{equation}\label{YosidaApprox}
    A^\la(z):=\frac 1\la \le[J_\la(z)-z\r]=A(J_\la(z)),\ \ \ \ z \in\,\mathcal{H}.
\end{equation}
By the Lipschitz continuity of $J_\lambda$, it is easy to check that 
\[\Vert A^\la(z_1)-A^\la(z_2)\Vert_{\mathcal{H}}\leq  \frac{2}{\la(1-\la \kappa)}\Vert z_1-z_2\Vert_{\mathcal{H}},\ \ \ \ z_1, z_2 \in\,\mathcal{H}.\]
Moreover, $A^\la$ is 
 quasi-dissipative in $\mathcal{H}$. Actually, by \eqref{gs1} and the definition of $A^\la$ in \eqref{YosidaApprox}, we have
\begin{equation}
\label{ge8}
\begin{aligned}
\langle A^\la(z_1)-A^\la(z_2),z_1-z_2\rangle _\mathcal{H}&=-\langle A^\la(z_1)-A^\la(z_2),(J_\lambda(z_1)-z_1)-(J_\lambda(z_2)-z_2)\rangle _\mathcal{H} \\[10pt]
&\quad +\langle A^\la(z_1)-A^\la(z_2),J_\lambda(z_1)-J_\lambda(z_2)\rangle _\mathcal{H}\\[10pt]
&\leq \kappa\Vert J_\lambda(z_1)-J_\lambda(z_2)\Vert_{\mathcal{H}}^2\\[10pt]
&\leq\frac{\kappa}{1-\la \kappa}\,\Vert z_1-z_2\Vert _{\mathcal{H}}^2,
\end{aligned}
\end{equation}
for any $z_1, z_2 \in\,\mathcal{H}$.
Moreover, as shown in \cite[Proposition 3.2]{Barbu}
for every $z \in\,D(A)$ we have
\begin{equation}
\label{gs4}
\Vert A^\la(z)\Vert _{\mathcal{H}}\leq \frac 1{1-\la \kappa}\Vert A(z)\Vert _{\mathcal{H}},
\end{equation}
and then
\begin{equation}
\label{gs121}
\Vert J_\la(z)-z\Vert _{\mathcal{H}}=\la \Vert A^\la(z)\Vert _{\mathcal{H}}\leq \frac \la{1-\la \kappa}\Vert A(z)\Vert _{\mathcal{H}}.
\end{equation}
Finally, as shown in \cite[Proposition 3.5]{Barbu}, we have
\begin{equation}
\label{gs5}
\lim_{\la\to 0}\Vert A^\la(z)-A(z)\Vert_{\mathcal{H}}=0,\ \ \ \ \ z \in\,D(A).
\end{equation}
Now we are ready to prove the main result of this section.
\begin{Theorem}
\label{teo4.2}
Under Assumptions \ref{Assumption1}, \ref{Assumption2} and \ref{Assumption3}, for every $(u_0, v_0) \in\,\mathcal{H}_1$ and every $T>0$ and $\mu>0$, there exists a unique adapted process $(u_\mu, v_\mu) \in\,L^2(\Omega,C([0,T],\mathcal{H}_1))$ which solves equation \eqref{gs52}.
 \end{Theorem}

\begin{proof}
Without loss of generality, we only consider $\mu=1$ here. As we have seen, the well-posedness of equation \eqref{gs52} is equivalent to the well-posedness of equation \eqref{SPDE3}. Therefore, here we deal with equation \eqref{SPDE3}. 

For every $\la \in\,(0,\bar{\la})$, we introduce the approximating problem
\begin{equation}
\label{SPDE3lambda}
dz_\la(t)=A^\la(z_\la(t))\,dt+\Sigma(z_\la(t))\,dw^Q(t),\ \ \ \ z_\la(0)=(u_0,v_0+g(u_0)).
\end{equation}
By Assumption \ref{Assumption1}, the mapping $\Sigma:\mathcal{H}\to \mathcal{L}_2(H_Q,\mathcal{H}_1)$ is bounded and Lipschitz continuous since 
\begin{equation}
    \Vert \Sigma(z) \Vert_{\mathcal{L}_2(H_Q,\mathcal{H}_1)} = \Vert \sigma(u) \Vert_{\mathcal{L}_2(H_Q,H)}
\end{equation}
for all $z=(u,\eta)\in \mathcal{H}$. Then, since $A^\la$ is Lipschitz continuous in $\mathcal{H}$, there exists a unique solution $$z_\la=(u_\la,\eta_\la) \in\,L^p(\Omega,C([0,T];\mathcal{H}))$$ 
to equation \eqref{SPDE3lambda}, for every $T>0$ and $p\geq 1$, using the classical fixed point theorem for contractions. Moreover, thanks to \eqref{ge8}, we have
\begin{align*}
    \frac{d}{dt}\,\mathbb{E}\Vert z_\lambda (t)\Vert_{\mathcal{H}}^2 &=2\mathbb{E}\left\langle A^\lambda(z_\lambda (t)),z_\lambda (t) \right\rangle_{\mathcal{H}} + \mathbb{E}\Vert \Sigma(z_\la(t))\Vert^2_{\mathcal{L}_2(H_Q,\mathcal{H})} \\[10pt]
    &=2\mathbb{E}\left\langle A^\lambda(z_\lambda (t))-A^\lambda (0),z_\lambda (t) \right\rangle_{\mathcal{H}} + 2\mathbb{E}\langle A^\lambda(0),z_\lambda (t) \rangle_{\mathcal{H}} + \mathbb{E}\Vert \Sigma(z_\la(t))\Vert^2_{\mathcal{L}_2(H_Q,\mathcal{H})}\\[10pt]
    &\leq \frac{2\kappa}{1-\lambda\kappa} \mathbb{E}\Vert z_\lambda (t)\Vert_{\mathcal{H}}^2+ \Vert A^\lambda (0)\Vert_{\mathcal{H}}^2 + \mathbb{E}\Vert z_\lambda (t)\Vert_{\mathcal{H}}^2 + c.
\end{align*}
Moreover, due to  \eqref{gs4},
\begin{equation*}
    \Vert A^\lambda (0)\Vert_{\mathcal{H}}^2\leq \frac{1}{(1-\lambda\kappa)^2}\Vert A(0)\Vert_{\mathcal{H}}^2 =\frac{1}{(1-\lambda\kappa)^2} \Vert f(0)\Vert_{H^{-1}}^2.
\end{equation*}
 Therefore, there exists a constant $c$, independent of $\la \leq \bar{\la}/2$, such that
\begin{equation*}
    \frac{d}{dt}\,\mathbb{E}\Vert z_\lambda (t)\Vert_{\mathcal{H}}^2\leq c\le(\mathbb{E}\Vert z_\lambda (t)\Vert_{\mathcal{H}}^2 +1\r).
\end{equation*}
By Gr\"onwall's inequality, this implies
 \begin{equation}
\label{gs9}
\sup_{\la \in\,(0,\bar{\la}/2)}\sup_{t \in\,[0,T]}\mathbb{E}\Vert z_\la(t)\Vert _{\mathcal{H}}^2<\infty.
\end{equation}
In the rest of the proof, for an arbitrary $z=(u,\eta) \in\,\mathcal{H}$, we denote $z_1=u$ and $z_2=\eta$.

\medskip
{\em Step 1.} There exists $c>0$  such that
\begin{equation}
\label{gs7}
\mathbb{E}\Vert z_\la(t)\Vert ^2_{\mathcal{H}_1}+\gamma_0\int_0^t\mathbb{E}\Vert J_\la(z_\la(t))_1\Vert ^2_{H^1}\,ds\leq \Vert z_0\Vert ^2_{\mathcal{H}_1}+c\int_0^t \mathbb{E}\Vert z_\la(s)\Vert ^2_{\mathcal{H}}\,ds+ct,
\end{equation}
for all $\la \in\,(0,\hat{\la})$, where
\[\hat{\lambda} :=\frac{\bar{\la}}2\wedge\frac{\gamma_0}{8\Vert f\Vert_{\tiny{\text{Lip}}}}.\]

{\em Proof of Step 1.} We apply the It\^o formula to 
\[K(z)=\Vert z\Vert ^2_{\mathcal{H}_1}=\Vert u\Vert _{H^1}^2+\Vert \eta\Vert ^2_H,\]
 and we get
\begin{equation}\label{gs10}
    \begin{aligned}
   dK(z_\la(t))&=  \langle A^\la(z_\la(t)),DK(z_\la(t))\rangle_{\mathcal{H}_1}dt+ \Vert \Sigma(z_\la(t))\Vert^2_{\mathcal{L}_2(H_Q,\mathcal{H}_1)}\,dt\\[10pt]
  & \quad + \langle D K(z_\la(t)),\Sigma(z_\la(t))dw^Q(t)\rangle_{\mathcal{H}_1}\\[10pt]
    &=2\le[\langle A^\la(z_\la(t))_1,(-\Delta)z_\la(t)_1\rangle_H+\langle A^\la(z_\la(t))_2,z_\la(t)_2\rangle_H\r]\,dt+ \Vert \sigma(z_\la(t)_1)\Vert^2_{\mathcal{L}_2(H_Q,H)}\,dt\\[10pt]
    &\quad +2\langle z_\la(t)_2,\si(z_\la(t)_1)\,dw^Q(t)\rangle_H\\[10pt]
    &=:2\Phi_\la(t)\,dt+  \Vert \sigma(z_\la(t)_1)\Vert^2_{\mathcal{L}_2(H_Q,H)} \,dt+2\langle z_\la(t)_2,\si(z_\la(t)_1)\,dw^Q(t)\rangle_H.
    \end{aligned}
\end{equation}
Recall that $A^\la(z)=A(J_\la(z))=\frac{1}{\la}[J_\la(z)-z]$, which implies  that for every $z \in\,\mathcal{H}$
\[\begin{cases}
\ds{J_\la(z)_1+\la g(J_\la(z)_1)-\la J_\la(z)_2=z_1,}\\[10pt]
\ds{J_\la(z)_2-\la \Delta (J_\la(z)_1)-\la f(J_\la(z)_1)=z_2.}
\end{cases}\]
Therefore, we have
\begin{align*}
    \Phi_\la &=\langle -g(J_\la(z_\la )_1)+J_\la(z_\la )_2,(-\Delta)\le[J_\la(z_\la )_1+\la g(J_\la(z_\la )_1)-\la J_\la(z_\la )_2\r]\rangle_H\\[10pt]
    &\quad +\langle \Delta (J_\la(z_\la )_1)+ f(J_\la(z_\la )_1),J_\la(z_\la )_2-\la \Delta (J_\la(z_\la )_1)-\la f(J_\la(z_\la )_1)\rangle_H\\[10pt]
    &=-\langle\gamma(J_\la(z_\la )_1)\nabla  (J_\la(z_\la )_1),\nabla (J_\la(z_\la )_1)\rangle_H-\la \Vert g(J_\la(z_\la )_1)\Vert ^2_{H^1}-\la \Vert J_\la(z_\la )_2\Vert ^2_{H^1}\\[10pt]
    &\quad -\la \Vert J_\la(z_\la )_1\Vert ^2_{H^2}-\la \Vert f(J_\la(z_\la )_1)\Vert _H^2-2\la\langle \Delta (J_\la(z_\la )_1),f(J_\la(z_\la )_1)\rangle_H\\[10pt]
    &\quad +\langle f(J_\la(z_\la )_1),J_\la(z_\la )_2\rangle_H\\[10pt]
    & < -\gamma_0\,\Vert J_\la(z_\la )_1\Vert ^2_{H^1} -2\la\langle \Delta (J_\la(z_\la )_1),f(J_\la(z_\la )_1)\rangle_H+\langle f(J_\la(z_\la )_1),J_\la(z_\la )_2\rangle_H,
\end{align*}
where the last inequality uses Assumption \ref{Assumption2}.
Since $f:\mathbb{R}\to \mathbb{R}$ is Lipschitz continuous, we have $f\circ u \in\,H^1$, for every $u \in\,H^1$ and
\[\Vert f\circ u\Vert_{H^1}\leq \Vert f\Vert_{\tiny{\text{Lip}}}\Vert u\Vert_{H^1},\]
which implies
\begin{equation*}
    \vert\langle \Delta (J_\la(z_\la )_1),f(J_\la(z_\la )_1)\rangle_H\vert \leq \Vert J_\la(z_\la )_1 \Vert_{H^1}\Vert f(J_\la(z_\la )_1) \Vert_{H^1}\leq \Vert f\Vert_{\tiny{\text{Lip}}}\Vert J_\la(z_\la )_1 \Vert_{H^1}^2,
\end{equation*}
and
\begin{equation*}
    \le|\langle f(J_\la(z_\la )_1),J_\la(z_\la )_2\rangle_H\r|\leq \Vert f(J_\la(z_\la )_1)\Vert _{H^1}\Vert J_\la(z_\la)_2\Vert_{H^{-1}}\leq \Vert f\Vert_{\tiny{\text{Lip}}}\,\Vert J_\la(z_\la )_1\Vert _{H^1}\Vert J_\la(z_\la)_2\Vert_{H^{-1}}.
\end{equation*}
Then, thanks to Young's inequality, we  get
\begin{equation}\label{gs30}
    \begin{aligned}
    \Phi_\la &< -\gamma_0\,\Vert J_\la(z_\la )_1\Vert ^2_{H^1} +2\lambda \Vert f\Vert_{\tiny{\text{Lip}}}\Vert J_\la(z_\la )_1 \Vert_{H^1}^2+ \Vert f\Vert_{\tiny{\text{Lip}}}\,\Vert J_\la(z_\la )_1\Vert _{H^1}\Vert J_\la(z_\la)_2\Vert_{H^{-1}}\\[10pt]
    &\leq -\frac{\gamma_0}2\,\Vert J_\la(z_\la )_1\Vert ^2_{H^1}+c\,\Vert J_\la(z_\la)_2\Vert_{H^{-1}}^2,
    \end{aligned}
\end{equation}
for $\lambda \in\,(0, \gamma_0/(8\Vert f\Vert_{\tiny{\text{Lip}}}))$. Therefore, if we integrate \eqref{gs10} in time, apply \eqref{gs30}, and use both the boundedness of $\Vert \sigma(u)\Vert^2_{\mathcal{L}_2(H_Q,H)}$ and the Lipschitz continuity of $J_\la$ on $\mathcal{H}$, we obtain 
\begin{equation}\label{gx001}
\begin{aligned}
&\Vert z_\la(t)\Vert_{\mathcal{H}_1}^2+\gamma_0\int_0^t \Vert J_\la(z_\la(s) )_1\Vert ^2_{H^1}\,ds\\[10pt]
&\leq \Vert z_0\Vert_{\mathcal{H}_1}^2+c\int_0^t \Vert z_\la(s)\Vert_{\mathcal{H}}^2\,ds+\si_\infty^2\,t+2\int_0^t \langle z_\la(s)_2,\si(z_\la(s)_1)dw^Q(s)\rangle_H,
\end{aligned}
\end{equation}
from which we can derive \eqref{gs7} after taking expectation.

\medskip

{\em Step 2.}
There exists $c_T>0$ such that for every $\la \in\,(0,\hat{\la}/2)$
\begin{equation}
\label{gs1000}
\mathbb{E}\sup_{t \in\,[0,T]}\Vert z_\la(t)\Vert_{\mathcal{H}_1}^2\leq \Vert z_0\Vert_{\mathcal{H}_1}^2+c_T.
\end{equation}
{\em Proof of Step 2.} By taking the supremum in time for \eqref{gx001}, we have
\[
\begin{aligned}
& \mathbb{E}\sup_{t \in\,[0,T]}\Vert z_\la(t)\Vert_{\mathcal{H}_1}^2\\[10pt]
& \leq \Vert z_0\Vert_{\mathcal{H}_1}^2+c\int_0^T\mathbb{E} \Vert z_\la(s)\Vert_{\mathcal{H}}^2\,ds+c\,T+ 2\,\mathbb{E}\sup_{t\in[0,T]}\left\vert \int_0^t\langle z_\la(s)_2,\si(z_\la(s)_1)\,dw^Q(s)\rangle_H\right\vert\\[10pt]
& \leq \Vert z_0\Vert_{\mathcal{H}_1}^2+c\int_0^T\mathbb{E} \Vert z_\la(s)\Vert_{\mathcal{H}}^2\,ds+c\,T+c\le(\mathbb{E} \int_0^T \Vert  \si (z_\la(s)_1)\Vert_{\mathcal{L}_2(H_Q,H)}^2 \Vert z_\la(s)_2\Vert_{H}^2\,ds\r)^{\frac 12}\\[10pt]
& \leq \Vert z_0\Vert_{\mathcal{H}_1}^2+c\int_0^T\mathbb{E} \Vert z_\la(s)\Vert_{\mathcal{H}}^2\,ds+c_T+\frac 12\, \mathbb{E}\sup_{t \in\, [0,T]}\Vert z_\la(t)\Vert_{\mathcal{H}_1}^2,
\end{aligned}\]
where the last inequality follows from Assumption \ref{Assumption1} and Young's inequality. 
Due to \eqref{gs9}, this implies \eqref{gs1000}.

\medskip

{\em Step 3.} There exists $z \in\,L^\infty(0,T;L^2(\Omega,\mathcal{H}))$ such that
\begin{equation}
\label{gs6}
\lim_{\la\to 0}\,\sup_{t \in\,[0,T]}\,\mathbb{E}\Vert z_\la(t)-z(t)\Vert _{\mathcal{H}}^2=0.
\end{equation}

{\em Proof of Step 3.} For every $\la, \nu \in\,(0,\hat{\la}/2)$, we set 
\[\varrho_{\la, \nu}(t):=z_\la(t)-z_\nu(t),\ \ \ \ t \in\,[0,T].\] Then, we have
\[d\varrho_{\la, \nu}(t)=\le[A^\la(z_\la(t))-A^\nu(z_\nu(t))\r]\,dt+\le[\Sigma(z_\la(t))-\Sigma(z_\nu(t))\r]dw^Q(t),\ \ \ \ \varrho_{\la, \nu}(0)=0.\]
Together with \eqref{YosidaApprox}, i.e.,  $z=J_\lambda(z)-\lambda A(J_\lambda(z))$, we have
\begin{align*}
 d\,\Vert \varrho_{\la, \nu}(t)\Vert _{\mathcal{H}}^2 & =2\le[\langle A^\la(z_\la(t))-A^\nu(z_\nu(t)),\varrho_{\la, \nu}(t)\rangle_{\mathcal{H}}+\Vert \Sigma(z_\la(t))-\Sigma(z_\nu(t))\Vert^2_{\mathcal{L}_2(H_Q,\mathcal{H})}\r]\,dt\\[10pt]
 & \quad +2\langle \varrho_{\la, \nu}(t),\le[\Sigma(z_\la(t))-\Sigma(z_\nu(t))\r]dw^Q(t)\rangle_{\mathcal{H}}\\[10pt]
    &=\big[\langle A(J_\la(z_\la(t)))-A(J_\nu(z_\nu(t))),J_\la(z_\la(t))-J_\nu(z_\nu(t))\rangle_{\mathcal{H}}\\[10pt]
    &\quad - \langle A^\la(z_\la(t))-A^\nu(z_\nu(t)),\la A^\la(z_\la(t))-\nu A^\nu(z_\nu(t))\rangle_{\mathcal{H}}\\[10pt]
    & \quad  +\Vert \Sigma(z_\la(t))-\Sigma(z_\nu(t))\Vert^2_{\mathcal{L}_2(H_Q,\mathcal{H})}\big]\,dt+2\langle \varrho_{\la, \nu}(t),\le[\Sigma(z_\la(t))-\Sigma(z_\nu(t))\r]dw^Q(t)\rangle_{\mathcal{H}}\\[10pt]
    &\leq \big[\kappa \Vert J_\la(z_\la(t))-J_\nu(z_\nu(t))\Vert _{\mathcal{H}}^2+c\le(\la+\nu\r)\le(\Vert A^\la(z_\la(t))\Vert _{\mathcal{H}}^2+\Vert A^\nu(z_\nu(t))\Vert _{\mathcal{H}}^2\r)\\[10pt]
    & \quad + c\Vert \varrho_{\la, \nu}(t)\Vert _{\mathcal{H}}^2\big]\,dt+2\langle \varrho_{\la, \nu}(t),\le[\Sigma(z_\la(t))-\Sigma(z_\nu(t))\r]dw^Q(t)\rangle_{\mathcal{H}}\\[10pt]
    &\leq c\big[ \Vert \varrho_{\la, \nu}(t)\Vert _{\mathcal{H}}^2+c\le(\la+\nu\r)\le(\Vert z_\la(t)\Vert ^2_{\mathcal{H}_1}+\Vert z_\nu(t)\Vert ^2_{\mathcal{H}_1}+1\r)\big]\,dt\\[10pt]
    & \quad+2\langle \varrho_{\la, \nu}(t),\le[\Sigma(z_\la(t))-\Sigma(z_\nu(t))\r]dw^Q(t)\rangle_{\mathcal{H}},
\end{align*}
where the first inequality follows from Lemma \ref{lem: Boperator} and the Lipschitz continuity of $\Sigma$, and the second inequality follows from \eqref{gs11}, \eqref{YosidaApprox} and \eqref{gs4}. 
Therefore
\[\mathbb{E}\Vert \varrho_{\la, \nu}(t)\Vert _{\mathcal{H}}^2\leq c\int_0^t  \mathbb{E}\Vert \varrho_{\la, \nu}(s)\Vert _{\mathcal{H}}^2\,ds+c\int_0^t \le(\la+\nu\r)\le(\mathbb{E}\Vert z_\la(s)\Vert ^2_{\mathcal{H}_1}+\mathbb{E}\Vert z_\nu(s)\Vert ^2_{\mathcal{H}_1}+1\r)\,ds.\]
Thanks to Gr\"onwall's inequality, this yields 
\[\sup_{t \in\,[0,T]}\mathbb{E}\,\Vert \varrho_{\la, \nu}(t)\Vert _{\mathcal{H}}^2\leq c_T\le(\la+\nu\r)\int_0^T\le(\mathbb{E}\Vert z_\la(s)\Vert ^2_{\mathcal{H}_1}+\mathbb{E}\Vert z_\nu(s)\Vert ^2_{\mathcal{H}_1}+1\r)\,ds.\]
In view of  \eqref{gs1000}, we conclude that, for every sequence $(\la_n)_{n \in\,\mathbb{N}}$ converging to zero, the sequence $(z_{\la_n})_{n \in\,\mathbb{N}}$ is Cauchy in $L^\infty(0,T;L^2(\Omega;\mathcal{H}))$.
In particular, there exists $z \in\,L^\infty(0,T;L^2(\Omega;\mathcal{H}))$ such that \eqref{gs6} holds.
\medskip

{\em Step 4.} For every $t \in\,[0,T]$, we have
\begin{equation}
\label{gs1001}
z(t)=(u_0,g(u_0)+v_0)+\int_0^t A(z(s))\,ds+\int_0^t \Sigma(z(s))dw^Q(s).
\end{equation}
Moreover, $z \in\,L^2(\Omega;C([0,T];\mathcal{H}_1))$.

{\em Proof of Step 4.}
For every $t \in\,[0,T]$ we have
\begin{equation}
\label{gs12}
z_\la(t)=(u_0,g(u_0)+v_0)+\int_0^t A^\la(z_\la(s))\,ds+\int_0^t \Sigma(z_\la(s))dw^Q(s).\end{equation}
If we define  $\mathcal{H}_{-1}=H^{-1}\times H^{-2}$, we have
\[\Vert A(z_1)-A(z_2)\Vert _{\mathcal{H}_{-1}}\leq c\,\Vert z_1-z_2\Vert _{\mathcal{H}},\ \ \ \ z_1,\ z_2 \in\,\mathcal{H}.\]
Since $z(t) \in\,L^2(\Omega;\mathcal{H})$, by \eqref{gs11} and \eqref{gs4} this implies
\begin{align*}
    \Vert A^\la(z_\la(s))-A(z(s))\Vert _{\mathcal{H}_{-1}} &=\Vert A(J_\la(z_\la(s)))-A(z(s))\Vert _{\mathcal{H}_{-1}}\\[10pt]
    &\leq c\,\Vert J_\la(z_\la(s))-z_\la(s)\Vert _{\mathcal{H}}+c\,\Vert z_\la(s)-z(s)\Vert _{\mathcal{H}}\\[10pt]
    &\leq c\,\la (\Vert z_\la(s)\Vert _{\mathcal{H}_1}+1)+c\,\Vert z_\la(s)-z(s)\Vert _{\mathcal{H}}.
\end{align*}
Therefore
\[\begin{array}{l}
\ds{\mathbb{E}\sup_{t \in\,[0,T]}\,\le\Vert  \int_0^t A^\la(z_\la(s))\,ds-\int_0^t A(z(s))\,ds\r\Vert ^2_{\mathcal{H}_{-1}}}\\[15pt]
\ds{\leq c_T\int_0^T\le[\la\,( \mathbb{E}\Vert z_\la(s)\Vert ^2_{\mathcal{H}_1}+1)+\mathbb{E}\,\Vert z_\la(s)-z(s)\Vert ^2_{\mathcal{H}}\r]\,ds.}
\end{array}\]
Thanks to \eqref{gs1000} and \eqref{gs6}, this 
implies 
\begin{equation}
\label{gs2300}
\lim_{\la\to 0} \int_0^\cdot A^\la(z_\la(s))\,ds=\int_0^\cdot A(z(s))\,ds,\ \ \ \text{in}\ L^2(\Omega;C([0,T];\mathcal{H}_{-1})).
\end{equation}
Moreover, 
\[\begin{array}{ll}
\ds{\mathbb{E}\sup_{t \in\,[0,T]}\le\Vert\int_0^t \le[\Sigma(z_\la(s))-\Sigma(z(s))\r]dw^Q(s)\r\Vert^2_{\mathcal{H}}}  &  \ds{\leq c\int_0^T\Vert \Sigma(z_\la(s))-\Sigma(z(s))\Vert^2_{\mathcal{L}_2(H_Q,\mathcal{H})} ds}\\[15pt]
&\ds{\leq c\int_0^T \mathbb{E}\Vert z_\la(s)-z(s)\Vert^2_{\mathcal{H}} ds,}
\end{array}\]
so that
\[\lim_{\la\to 0} \int_0^\cdot \Sigma(z_\la(s))\,dw^Q(s)=\int_0^\cdot \Sigma(z(s))\,dw^Q,\ \ \ \text{in}\ L^2(\Omega;C([0,T];\mathcal{H})).\]
This, together with \eqref{gs2300}, allows us to conclude that for every $t \in\,[0,T]$ we can take the $L^2(\Omega,C([0,T];\mathcal{H}_{-1}))$-limit on both sides of \eqref{gs12}, as $\la$ goes to zero,  and we get
\[z(t)=(u_0,g(u_0)+v_0)+\int_0^t A(z(s))\,ds+\int_0^t\Sigma(z(s))dw^Q(t).\]
Moreover, by proceeding as for $z_\la$, we have that $z \in\,L^2(\Omega;L^\infty(0,T;\mathcal{H}_1))$.

Now, in order to prove the continuity of trajectories in $\mathcal{H}_1$, we denote by $\theta(t)$ the solution of the  problem
\[d\theta(t)=L \theta(t)\,dt+\Sigma(z(t))dw^Q(t),\ \ \ \ \ \theta(0)=(u_0,g(u_0)+v_0),\]
where 
$L(\theta_1,\theta_2)=(\theta_2,\Delta \theta_1)$. Due to Assumption \ref{Assumption1} and the fact that $\theta(0) \in\,\mathcal{H}_1$, we have that $\theta$ belongs to $L^2(\Omega;C([0,T];\mathcal{H}_1))$ (for a proof see \cite[Theorem 5.11]{DPZ}).
Now, if we define $\hat{z}(t):=z(t)-\theta(t)$, for $t \in\,[0,T]$, and $M(z_1,z_2)=(-g(z_1),f(z_1))$, for $(z_1,z_2) \in\,\mathcal{H}_1$, we have
\[\frac{d}{dt}\hat{z}(t)=L\hat{z}(t)+M(z(t)),\ \ \ \ \si(0)=0,\]
so that, by applying the variation of constants formula, we obtain
\[\hat{z}(t)=\int_0^t S(t-s) M(z(s))\,ds,\]
where $S(t)$ is the group generated by the operator $L$, endowed with Dirichlet boundary conditions, in $\mathcal{H}_1$. Since
\[\Vert M(z)\Vert_{\mathcal{H}_1}\leq c\ \le(\Vert z\Vert_{\mathcal{H}_1}+1\r),\]
and $z \in\,L^2(\Omega;L^\infty(0,T;\mathcal{H}_1))$, we have that $\hat{z} \in\,L^2(\Omega;C([0,T];\mathcal{H}_1))$. Then, as $z=\hat{z}+\theta$, we conclude that $z \in\,L^2(\Omega;C([0,T];\mathcal{H}_1))$.

{\em Step 5.} Uniqueness holds.

{\em Proof of Step 5.}
Let $z_1$  and $z_2$ be two solutions of equation \eqref{gs52}. If we define $\varrho(t):=z_1(t)-z_2(t)$,  we have 
\[\begin{array}{ll}
\ds{\Vert \varrho(t)\Vert ^2_{\mathcal{H}}}  &  \ds{=\int_0^t\le[\langle A(z_1(s))-A(z_2(s)),\varrho(t)\rangle_{\mathcal{H}}+\Vert \sigma(z_1(s))-\sigma(z_2(s))\Vert^2_{\mathcal{L}_2(H_Q,\mathcal{H})}\r]\,ds }\\[10pt]
&\ds{\quad +2\int_0^t \langle \varrho(s),\le[\Sigma(z_1(s))-\Sigma(z_2(s))\r]dw^Q\rangle_{\mathcal{H}}.}
\end{array}\]
Hence, by Lemma \ref{lem: Boperator} and Assumption \ref{Assumption1} we have
\[\mathbb{E}\,\Vert \varrho(t)\Vert ^2_{\mathcal{H}}\leq c\,\int_0^t \mathbb{E}\,\Vert \varrho(s)\Vert ^2_{\mathcal{H}}\,ds,\]
and this implies that $z_1=z_2$.
\end{proof}

\section{Energy estimates}
\label{sec5}

In the previous section we have proved that  for any  $\mu>0$ and any $T>0$ there is a unique solution $(u_\mu,\partial_t u_\mu) \in\,L^2(\Omega;C([0,T],\mathcal{H}_1))$ to  system \eqref{SPDE2}.
In this section, we prove some  bounds for  $(u_\mu,\partial_t u_\mu)$, which are uniform with respect to $ \mu$. 

As we  have already done in the proof of Theorem \ref{teo4.2}, if we apply the It\^o formula to equation   \eqref{SPDE2} and the function 
\[
K_\mu(u,v)=\Vert u\Vert ^2_{H^1}+\mu \Vert v\Vert _H^2,
\]
we have  
\begin{align*}
\frac12\, dK_\mu(u_\mu,\partial_t u_\mu) &=\Large[\langle (-\Delta) u_\mu(t),\partial_t u_\mu(t)\rangle_H+\langle \Delta u_\mu(t),\partial_t u_\mu(t)\rangle_H-\langle \gamma(u_\mu(t))\partial_t u_\mu(t),\partial_t u_\mu(t)\rangle_H\\[10pt]
    &\quad +\langle f(u_\mu(t),\partial_t u_\mu(t)\rangle_H+\frac 1{2\mu} \Vert \si(u_\mu(t))\Vert^2_{\mathcal{L}_2(H_Q,H)}\Large]\,dt\\[10pt]
    &\quad +\langle \partial_t u_\mu(t),\si(u_\mu(s))dw^Q(t)\rangle_H.
\end{align*}
This implies
\begin{equation}\label{de1bis}
    \begin{aligned}
    &\frac 12 d \le[\Vert  u_\mu(t)\Vert _{H^1}^2+\mu\, \Vert  \partial_t u_\mu(t)\Vert _H^2 \r]\\[10pt]
    &=\le(\langle f(u_\mu(t),\partial_t u_\mu(t)\rangle_H-\langle \gamma(u_\mu(t))\partial_t u_\mu(t),\partial_t u_\mu(t)\rangle_H +\frac{1}{2\mu}\Vert \si(u_\mu(t))\Vert^2_{\mathcal{L}_2(H_Q,H)}\r)dt\\[10pt]
    &\quad +\langle \partial_t u_\mu(t),\si(u_\mu(t))dw^Q(t)\rangle_H\\[10pt]
    &\leq \le(c\,\le(\Vert  u_\mu(t)\Vert _{H}^2 +1\r)-\frac {\gamma_0}2 \Vert  \partial_t u_\mu(t)\Vert _H^2+\frac{\si_\infty^2}{2\mu}\r)dt+\langle \partial_t u_\mu(t),\si(u_\mu(t))dw^Q(t)\rangle_H.
\end{aligned}
\end{equation}
In particular,
\begin{equation}\label{de1}
    \begin{aligned}
    &\frac 12 \frac d{dt} \le[\mathbb{E}\Vert  u_\mu(t)\Vert _{H^1}^2+\mu\, \mathbb{E}\Vert  \partial_t u_\mu(t)\Vert _H^2 \r]\\[10pt]
&\leq -\frac{\gamma_0}{2\mu}\le[\mathbb{E}\Vert  u_\mu(t)\Vert _{H^1}^2+\mu\,\mathbb{E} \Vert  \partial_t u_\mu(t)\Vert _H^2 -\bar{c}\r]+c\le(\frac 1\mu \mathbb{E}\Vert  u_\mu(t)\Vert _{H^1}^2+1\r),
    \end{aligned}
\end{equation}
where $\bar{c} =\si_\infty^2 /\gamma_0$. Moreover, we have  
\begin{equation}\label{de2}
\frac{d}{dt}\, \mu \Vert  u_\mu(t)\Vert _{H}^2=2\mu \langle u_\mu(t),\partial_t u_\mu(t)\rangle_H,
\end{equation}
and 
\begin{equation}\label{de3}
\begin{aligned}
 \mu\,   d \langle u_\mu(t),\partial_t u_\mu(t)\rangle_H &=\le[\mu\,  \Vert  \partial_t u_\mu(t)\Vert _H^2-\Vert  u_\mu(t)\Vert _{H^1}^2 -\langle u_\mu(t),\gamma(u_\mu(t))\partial_t u_\mu(t)\rangle_H\r.\\[10pt]
    &\quad \le.+\langle f(u_\mu(t)),u_\mu(t)\rangle_H\r]\,dt+\langle u_\mu(t),\si(u_\mu(t))dw^Q(t)\rangle_H.
\end{aligned}
\end{equation}
\begin{Lemma}\label{lem:tightness}
Under Assumptions \ref{Assumption1}, \ref{Assumption2} and \ref{Assumption3}, for every $T>0$ and $(u_0, v_0) \in\,\mathcal{H}_1$ there exists some constant $c_T=c_T(\Vert u_0\Vert _{H^1}, \Vert v_0\Vert _H)$, independent of $\mu$, such that 
\begin{equation}
\label{gs24}
\begin{aligned}
&\mathbb{E}\sup_{r \in\,[0,t]}\Vert  u_\mu(r) \Vert _H^2 +\int_0^t \mathbb{E}\Vert  u_\mu(s)\Vert _{H^1}^2 ds\\[10pt]
& \leq c_T \le(1+  \mu \int_0^t \mathbb{E}\Vert  \partial_t u_\mu(s)\Vert _{H}^2 ds +\mu^2\,  \mathbb{E}\sup_{r \in\,[0,t]}\Vert  \partial_t u_\mu(r)\Vert _{H}^2\r), 
\end{aligned}
\end{equation}
for every $ \mu \in\,(0,1)$
and $t \in\,[0,T]$.
\end{Lemma}

\begin{proof}
We define 
\[
\Gamma(r)=\int_0^r x\gamma(x) dx,\ \ \ \ r \in\,\mathbb{R},
\]
and \[\Lambda(u)=\int_{\mathcal{O}}\Gamma(u(x))\,dx,\ \ \ \ u \in\,H.\]
It is easy to see that \eqref{nonlinearity assumption} implies 
\begin{equation}
0\leq \frac{\gamma_0}{2}r^2\leq \Gamma(r)\leq \frac{\gamma_1}{2}r^2,\ \ \ \ \ r \in\,\mathbb{R},
\end{equation}
so that, for every $u \in\,H$
\begin{equation}\label{inequality8}
0\leq \frac{\gamma_0}{2} \Vert  u \Vert _H^2 \leq \Lambda(u) \leq \frac{\gamma_1}{2} \Vert  u \Vert _H^2.
\end{equation}
Moreover, if $v(t)=\partial_t u(t) $, we have
\begin{equation}
\frac{d}{dt}\,\Lambda(u(t))   = \int_{\mathcal{O}}\gamma(u(t,x))u(t,x)v(t,x)\,dx=\langle u(t),\gamma(u(t))\partial_t u(t)\rangle_H.
\end{equation}
Therefore, thanks to \eqref{de3} and \eqref{inequality8}, this gives
\begin{align*}
    \frac{\gamma_0}{2} \Vert  u_\mu(t) \Vert _H^2 &\leq \Lambda(u_\mu(t))\\[10pt]
    &=  \Lambda(u_0) -\mu\,  \langle u_\mu(t) , \partial_t u_\mu(t)\rangle_H+ \mu\, \langle u_0,v_0\rangle_H + \mu \int_0^t \Vert  \partial_t u_\mu(s)\Vert _{H}^2 ds\\[10pt]
    & \quad  -\int_{0}^t \Vert  u_\mu(s)\Vert _{H^1}^2 ds+\int_0^t   \langle f(u_\mu(s)),u_\mu(s) \rangle_H ds+\int_0^t\langle u_\mu(s),\si(u_\mu(s))dw^Q(s)\rangle_H.
\end{align*}
In particular, for every $\mu \in\,(0,1)$, we have
\begin{align*}
    \frac{\gamma_0}{4} \Vert  u_\mu(t) \Vert _H^2 &\leq c+ c\mu^2\Vert \partial_t u_\mu(t)\Vert _H^2+  \mu \int_0^t \Vert  \partial_t u_\mu(s)\Vert _{H}^2 ds \\[10pt]
    &\quad -\int_{0}^t\Vert  u_\mu(s)\Vert _{H^1}^2 ds+\int_0^t   \langle f(u_\mu(s)) ,  u_\mu(s)\rangle_H ds+\int_0^t\langle u_\mu(s),\si(u_\mu(s))dw^Q(s)\rangle_H.
\end{align*}
Now, by \eqref{Poincare} and \eqref{gs25}, we have
\begin{align*}
\langle f(u_\mu(s)) ,  u_\mu(s)\rangle_H &\leq \la\,\Vert u_\mu(s)\Vert _H^2+c\le(1+\Vert u_\mu(s)\Vert ^{1+\d}_{L^{1+\d}(\mathcal{O})}\r)\\[10pt]
    &\leq \le[\la/\a_1+\le(1-\la/\a_1\r)/2\r]\Vert u_\mu(s)\Vert _{H^1}^2 +c,
\end{align*}
so that
\begin{align*}
  &  \frac{\gamma_0}{4}\Vert  u_\mu(t) \Vert _H^2 +\frac 12\le(1-\la/\a_1\r)\,\int_{0}^t \Vert  u_\mu(s)\Vert _{H^1}^2 ds\\[10pt]
&\leq c_T+c\mu^2 \Vert \partial_t u_\mu(t)\Vert _H^2+  \mu \int_0^t  \Vert  \partial_t u_\mu(s)\Vert _{H}^2 ds+\int_0^t\langle u_\mu(s),\si(u_\mu(s))dw^Q(s)\rangle_H.
\end{align*}
This implies
\begin{align*}
  &  \mathbb{E}\sup_{r \in\,[0,t]}\Vert  u_\mu(r) \Vert _H^2 +\int_{0}^t \mathbb{E}\Vert  u_\mu(s)\Vert _{H^1}^2 ds\\[10pt]
&\leq c_T+c\mu^2 \mathbb{E}\sup_{r \in\,[0,t]}\Vert \partial_t u_\mu(r)\Vert _H^2+ c \mu \int_0^t  \mathbb{E}\Vert  \partial_t u_\mu(s)\Vert _{H}^2 ds\\[10pt]
&\quad +c\,\mathbb{E}\sup_{r \in\,[0,t]}\le|\int_0^r\langle u_\mu(s),\si(u_\mu(s))dw^Q(s)\rangle_H\r|\\[10pt]
&\leq c_T+c\mu^2 \mathbb{E}\sup_{r \in\,[0,t]}\Vert \partial_t u_\mu(r)\Vert _H^2+ c \mu \int_0^t  \mathbb{E}\Vert  \partial_t u_\mu(s)\Vert _{H}^2 ds+\frac 12 \int_0^t \mathbb{E}\Vert u_\mu(s)\Vert_H^2\,ds,
\end{align*}
and \eqref{gs24} follows.

\end{proof}

\begin{Proposition}
\label{lemma100}
Under Assumptions \ref{Assumption1}, \ref{Assumption2} and \ref{Assumption3}, for every $T>0$ and $(u_0, v_0) \in\,\mathcal{H}_1$ there exist some constants $c_T$ and $\mu_T>0$ depending on $\Vert u_0\Vert _{H^1}, \Vert v_0\Vert _H$ such that 
\begin{equation}
\label{gs24tris}
\begin{aligned}
& \mathbb{E}\sup_{r \in\,[0,T]}\Vert  u_\mu(r) \Vert _{H^1}^2 +\mu\, \mathbb{E}\sup_{r \in\,[0,T]}\Vert   \partial_t u_\mu(r) \Vert _{H}^2+ \int_0^T \mathbb{E}\Vert  \partial_t u_\mu(s)\Vert _{H}^2 ds\leq \frac{c_T}\mu,\end{aligned}
\end{equation}
for every $ \mu \in\,(0,\mu_T)$.
\end{Proposition}

\begin{proof}
Due to \eqref{de1bis}, for every $\mu \in\,(0,1)$ we have
\begin{align*}
&\Vert  u_\mu(t) \Vert _{H^1}^2 +\mu\, \Vert   \partial_t u_\mu(t) \Vert _{H}^2+\frac{\gamma_0}2\int_0^t \Vert \partial_t u_\mu(s) \Vert _{H}^2\,ds\\[10pt]
&\leq \frac{c_T}\mu+c\,\int_0^t \Vert u_\mu(s) \Vert _{H}^2\,ds+\int_0^t\langle \partial_t u_\mu(s),\si(u_\mu(s))dw^Q(s)\rangle_H.
\end{align*}
This implies 
\begin{align*}
&\, \mathbb{E}\sup_{r \in\,[0,t]}\Vert  u_\mu(r) \Vert _{H^1}^2 +\mu\,\mathbb{E} \sup_{r \in\,[0,t]}\Vert   \partial_t u_\mu(r) \Vert _{H}^2+\frac{\gamma_0}2\int_0^t \mathbb{E}\Vert \partial_t u_\mu(s) \Vert _{H}^2\,ds\\[10pt]
&\leq \frac{c_T}\mu+c\,\int_0^t \mathbb{E}\Vert u_\mu(s) \Vert _{H}^2\,ds+\,\mathbb{E}\sup_{r \in\,[0,t]}\le|\int_0^r\langle \partial_t u_\mu(s),\si(u_\mu(s))dw^Q(s)\rangle_H\r|\\[10pt]
&\leq \frac{c_T}\mu+c\,\int_0^t \mathbb{E}\Vert u_\mu(s) \Vert _{H}^2\,ds+\frac{\gamma_0}4\,\int_0^t\mathbb{E}\Vert \partial_t u_\mu(s) \Vert _{H}^2\,ds.
\end{align*}
Therefore, we can conclude the proof by using  \eqref{gs24}.
\end{proof}

\begin{Remark}
{\em Combining \eqref{gs24} and \eqref{gs24tris}, we obtain that for every $T>0$ there exist $c_T, \mu_T>0$ such that
\begin{equation}
\label{gs24quater}
\mathbb{E}\sup_{r \in\,[0,T]}\Vert  u_\mu(r) \Vert _H^2 +\int_0^T \mathbb{E}\Vert  u_\mu(s)\Vert _{H^1}^2 ds\leq c_T
\end{equation}
for all $\mu \in\,(0,\mu_T)$.}
\end{Remark}

In fact, we can prove a better bound for the $L^2(\Omega;L^\infty(0,T;\mathcal{H}_1))$-norm of $(u_\mu,\sqrt{\mu}\,\partial_t u_\mu)$ than the one in \eqref{gs24tris}.

\begin{Proposition}\label{Energy estimate}
Under Assumptions \ref{Assumption1}, \ref{Assumption2} and \ref{Assumption3}, given any  $T>0$, there exist  $c_T, \mu_T>0$ such that  for all $\mu \in\,(0,\mu_T)$
\begin{equation}
\label{[2]}
\sqrt{\mu}\ \mathbb{E}\sup_{t\in[0,T]}\le(\Vert  u_\mu(t)\Vert _{H^1}^2+\mu\, \Vert  \partial_t u_\mu(t) \Vert _H^2\r) \leq c_T.
\end{equation}

\end{Proposition}

\begin{proof} 
Assume \eqref{[2]} is not true. Then, we can find  a sequence  $(\mu_k)_{k \in\,\mathbb{N}}\subset (0,1)$ converging to $0$, as $k\rightarrow \infty$, such that 
\begin{equation}\label{inequality12}
    \lim_{k\rightarrow \infty} \sqrt{\mu_k}\ \mathbb{E} \sup_{t\in[0,T]} \left( \Vert  u_{\mu_k}(t)\Vert _{H^1}^2 +\mu_k \Vert  \partial_t u_{\mu_k}(t) \Vert _H^2\right)=\infty.
\end{equation}
In what follows, to simplify our notation we define 
\[L_k(t):=\Vert  u_{\mu_k}(t)\Vert _{H^1}^2+\mu_k \Vert  \partial_tu_{\mu_k}(t) \Vert _H^2,\ \ \ \ t \in\,[0,T].\]
By Theorem \ref{teo4.2}, all $L_k(t)$ are continuous in $t$, $\mathbb{P}$-a.s.. Therefore, for every $k \in\,\mathbb{N}$ there exist a random time $t_k \in\,[0,T]$ such that 
\[L_k(t_k)=\sup_{t\in[0,T]}L_k(t).\]
For any random time $s$ such that $\mathbb{P}(s\leq t_k)=1$, from \eqref{de1bis} we have 
\begin{equation}\label{gx002}
L_k(t_k)-L_k(s) \leq \int_s^{t_k} \le(c(\Vert u_{\mu_k}(\tau) \Vert_H^2 +1)+\frac{\sigma^2_\infty}{\mu_k} \r) d\tau +M_k(t_k)-M_k(s),    
\end{equation}
where 
\[M_k(t)=\int_0^t \langle \partial_t u_{\mu_k}(s),\si(u_{\mu_k}(s))dw^Q(s)\rangle_H.\]
If we define the random variables
\[M_k^* :=\sup_{t\in[0,T]} \vert M_k(t)\vert,\quad U_k^* := \sup_{t\in[0,T]} \Vert u_{\mu_k}(t)\Vert_H^2,\]
then by Proposition \ref{lemma100} and \eqref{gs24quater} we have 
\begin{equation}\label{gx003}
    \mathbb{E}(M_k^*)\leq c\le( \int_0^T \mathbb{E} \Vert \partial_t u_{\mu_k}(t)\Vert_H^2 dt\r)^{\frac{1}{2}}\leq \frac{c_T}{\sqrt{\mu_k}}, \quad \mathbb{E}(U_k^*)\leq c_T.
\end{equation}
Due to the definition of $M_k^*$ and $U_k^*$, there exists a constant $\la_T$  independent of $k$ such that $\Vert u_0\Vert_{H^1}^2+\Vert v_0\Vert_{H}^2\leq \la_T$ and 
\[L_k(t_k)-L_k(s) \leq \la_T(U_k^*+1) + \frac{\sigma^2_\infty(t_k -s)}{\mu_k} +2M_k^*.  \]
If we take $s=0$, then
\[t_k\geq \frac{\mu_k}{\sigma^2_\infty}(L_k(t_k)-\la_T-\la_T(U_k^*+1)-2M_k^*)=:\frac{\mu_k }{\sigma^2_\infty}\,\theta_k.\]
On the set $E_k:=\{\theta_k>0\}$ we consider $s\in[t_k-\frac{\mu_k}{2\sigma^2_\infty}\,\theta_k,t_k]$, for which we have
\[L_k(s)\geq  L_k(t_k) - \la_T(U_k^*+1) - \frac{\theta_k}{2} - 2M_k^* = \frac{1}{2}[(L_k(t_k)-\la_T(U_k^*+1) - 2M_k^*)+\la_T].\]
Finally, if we define
\[I_k := \int_{0}^{T} L_k(s)ds = \int_0^T \le(\Vert u_{\mu_k}(s)\Vert_{H^1}^2+ \mu_k \Vert \partial_t u_{\mu_k}(s)\Vert_H^2\r)\,ds , \]
we have
\begin{equation*}
    I_k \geq \int_{t_k-\frac{\mu_k}{2\sigma^2_\infty}\,\theta_k}^{t_k} L_k(s)ds 
    \geq \frac{\mu_k}{4\sigma_\infty^2}\le[(L_k(t_k)-\la_T(U_k^*+1) - 2M_k^*)^2-\la_T^2\r]
\end{equation*}
on $E_k$.
Thus, by taking expectation on both sides, we get 
\begin{equation}\label{gx007}
\begin{aligned}
    \mathbb{E}(I_k) &\geq \mathbb{E}(I_k \,;\,E_k)\geq \mathbb{E}\le(\frac{\mu_k}{4\sigma_\infty^2}(L_k(t_k)-\la_T(U_k^*+1) - 2M_k^*)^2 \,;\,E_k\r)-\frac{\mu_k \la_T^2}{4\sigma_\infty^2}. 
\end{aligned}
\end{equation}
By \eqref{inequality12} and \eqref{gx003}, we know 
\[\lim_{k\rightarrow \infty}\sqrt{\mu_k}\ \mathbb{E}\theta_k=\infty.\] Moreover, 
\begin{equation}\label{gx008}
    \mathbb{E}(\sqrt{\mu_k}\theta_k)\leq \mathbb{E}(\sqrt{\mu_k}\theta_k \,;\,E_k)\leq \mathbb{E}(\sqrt{\mu_k}(\theta_k+\la_T)\,;\,E_k)\leq \le[\mathbb{E}(\mu_k (\theta_k+\la_T)^2\,;\,E_k)\r]^{\frac{1}{2}}.
\end{equation}
Combine \eqref{gx007} and \eqref{gx008}, we have 
\begin{equation*}
    \mathbb{E}(I_k)\geq \frac{1}{4\si_\infty^2}\mathbb{E}(\mu_k (\theta_k+\la_T)^2\,;\,E_k)-\frac{\mu_k \la_T^2}{4\sigma_\infty^2}\geq \frac{1}{4\si_\infty^2}\le(\sqrt{\mu_k}\ \mathbb{E}\theta_k\r)^2-\frac{\mu_k \la_T^2}{4\sigma_\infty^2}.
\end{equation*}
This implies that $\lim_{k\to\infty} \mathbb{E}(I_k)=\infty$, which contradicts to \eqref{gs24tris} and \eqref{gs24quater}. Therefore, \eqref{[2]}  must be true and the proof is complete. 
\end{proof}

\section{Tightness}
\label{sec6}

For every $\mu>0$ and $T>0$, we shall  define  
\[
    \rho_\mu(t)=g\left(u_\mu(t)\right),\ \ \ \ t \in\,[0,T].
\]
In this section, we study the tightness of  the family of measures $(\mathcal{L}(\rho_{\mu_k}))_{k \in\,\mathbb{N}}$, for any sequence $(\mu_k)_{k \in\,\mathbb{N}}$ converging to zero. According to Assumption \ref{Assumption2} and the definition of $g$, we know that 
\[\vert g(r)\vert \leq \gamma_1\vert r\vert,\ \ \ \ \vert g'(r)\vert \leq \gamma_1,\ \ \ \ \ r \in \,\mathbb{R}.\] Therefore, for every $\mu>0$ and $t \in\,[0,T]$
\[
\Vert \rho_\mu(t)\Vert_H\leq \gamma_1\Vert u_\mu(t)\Vert_H,\ \ \ \ \ \Vert \rho_\mu(t)\Vert_{H^1}\leq \gamma_1\Vert u_\mu(t)\Vert_{H^1}.\]
As a consequence of \eqref{gs24quater} and \eqref{[2]}, this implies that there exist $c_T, \mu_T>0$ such that
\begin{equation}
\label{cor2}
\mathbb{E} \sup_{t\in[0,T]}\le( \Vert \rho_\mu(t)\Vert_{H}^2+\sqrt{\mu}\,\Vert \rho_\mu(t)\Vert_{H^1}^2\r)+\mathbb{E}\int_0^T\Vert \rho_\mu(s)\Vert_{H^1}^2\,ds\leq c_T,\ \ \ \ \mu \in\,(0,\mu_T).
\end{equation}
Since $g(r)$ is a strictly increasing function, it is invertible and for every $\mu>0$
$$u_\mu(t)=g^{-1}(\rho_\mu(t)),\ \ \ \ \ t \in\,[0,T].$$
This implies that
\begin{align*}
    \Delta u_\mu(t) =\text{div}\left[\nabla g^{-1}(\rho_\mu(t)) \right]
    = \text{div}\left[ \frac{1}{\gamma(g^{-1}(\rho_\mu(t)))}\nabla\rho_\mu(t) \right].
\end{align*}
Moreover, by the definition of $\rho_\mu$, we have that \begin{equation}
    \nabla \rho_\mu(t)=\gamma(u_\mu(t))\nabla u_\mu(t),\qquad \partial_t  \rho_\mu(t)=\gamma(u_\mu(t))\partial_t u_\mu(t).
\end{equation}
Therefore, if we define
\begin{equation}
\label{gs2031}
b(r):=\frac{1}{\gamma(g^{-1}(r))},\ \ \ \ \ F(r):=f(g^{-1}(r)),\ \ \ \ r \in\,\mathbb{R},\end{equation}
and
\begin{equation}
\label{gs2032}
\si_g(h):=\si(g^{-1}\circ h),\ \ \ \ \ h \in\,H,
\end{equation}
we can rewrite equation \eqref{SPDE2} into the following form
\begin{equation}\label{rhoEquation}
\begin{aligned}
    \rho_\mu(t)+\mu \partial_t u_\mu(t)& =g(u_0)+\mu v_0 +\int_0^t \text{div}[b(\rho_\mu(s))\nabla \rho_\mu(s)] ds\\[10pt]
    &\quad + \int_0^t F(\rho_\mu(s)) ds+ \int_0^t\si_g(\rho_\mu(s))dw^Q(s)
    \end{aligned}
\end{equation}
in space $H^{-1}$. From Assumption \ref{Assumption2}, we have
\begin{equation}\label{eq:AUniform}
   0<\frac 1{\gamma_1}\leq b(r)\leq \frac 1{\gamma_0},\ \ \ \ r \in\,\mathbb{R}. 
\end{equation}

In what follows, we shall define
\begin{equation}
\label{gs1003}
X_1:=C\big([0,T];\bigcap_{\d>0} H^{-\d}\big),\ \ \ \ \ X_2:=\bigcap_{p<\infty} L^p(0,T;H).
\end{equation}
Both spaces turn out to be complete and separable metric spaces, endowed with the distances
\begin{equation}
\label{d1}
d_{X_1}(x,y)=\sum_{n=1}^\infty \frac1{2^n}\le(|x-y|_{C([0,T];H^{-\frac{1}{n}})}\wedge 1\r),\end{equation}
and
\begin{equation}
\label{d2}
d_{X_2}(x,y)=\sum_{n=1}^\infty \frac1{2^n}\le(|x-y|_{L^n(0,T;H)}\wedge 1\r).\end{equation}
Notice that both $X_1$ and $X_2$ contain $L^\infty(0,T,H)$, with proper inclusion.

\begin{Theorem}\label{theorem:tightness}
Assume that Assumptions \ref{Assumption1}, \ref{Assumption2} and \ref{Assumption3} are satisfied, and fix  any initial datum $(u_0,v_0)\in \mathcal{H}_1$ and any $T>0$. Then,  for any sequence $(\mu_k)_{k \in\,\mathbb{N}}$ converging to zero,  the family of   probability measures 
\[(\mathcal{L}(g(u_{\mu_k})))_{k \in\,\mathbb{N}}\subset \mathcal{P}(X_1\cap X_2),\]
is tight.
\end{Theorem}
\begin{proof}
For every $\theta \in\,(0,1)$, let   $C^\theta([0,T];H^{-1})$  denote the space of $\theta$-H\"older continuous functions  defined on $[0,T]$ with values in $H^{-1}$.
As a first step, we prove that there exists some $\theta \in\,(0,1)$ such that the family  
\[(\rho_\mu+\mu \partial_t u_\mu)_{\mu \in\,(0,\mu_T)}\subset L^1(\Omega;C^\theta([0,T];H^{-1}))\]
 is bounded. For the first integral term in \eqref{rhoEquation}, given any $0\leq t_1<t_2 \leq T$, by \eqref{cor2}  and \eqref{eq:AUniform} we have
\begin{equation}\label{holder2}
   \begin{aligned}
\mathbb{E}    \int_{t_1}^{t_2} \Vert \text{div}[b(\rho_\mu(s))\nabla \rho_\mu(s)] \Vert_{H^{-1}} ds & \leq c\, \mathbb{E}\int_{t_1}^{t_2} \Vert b(\rho_\mu(s))\nabla \rho_\mu(s) \Vert_H  ds\\[0pt]
    &\leq \frac{c}{\gamma_0}(t_2 - t_1)^{\frac{1}{2}} \left(\, \int_{0}^{T} \mathbb{E}\Vert \rho_\mu (s) \Vert _{H^1}^2 ds  \right)^{\frac{1}{2}}\\[10pt]
    &\leq \frac{c}{\gamma_0}(t_2 - t_1)^{\frac{1}{2}}.
\end{aligned} 
\end{equation}
For the second integral term in \eqref{rhoEquation}, thanks to \eqref{gs24quater} we have
\begin{equation}\label{holder3}
 \mathbb{E}\,   \int_{t_1}^{t_2}\Vert F(\rho_\mu(s))\Vert_{H^{-1}} ds = \mathbb{E}\,  \int_{t_1}^{t_2}\Vert f(u_\mu(s))\Vert_{H^{-1}} ds \leq c\,   \int_{t_1}^{t_2}\le(\mathbb{E}\, \Vert u_\mu(s)\Vert_{H}+1\r)\, ds \leq c(t_2 -t_1).
\end{equation}
Finally, due to the boundedness of $\si_g$, by proceeding as in \cite[Theorem 5.11 and Theorem 5.15]{DPZ}, by using a factorization argument we have that 
\[\sup_{\mu>0}\mathbb{E}\le\Vert \int_0^\cdot \si(u_\mu(s))dw^Q(s)\r\Vert_{C^\theta(0,T;H^{-1})}<\infty.\]
for any $\theta \in\,(0,1/2)$. 
Therefore, by putting this together with  \eqref{holder2} and \eqref{holder3},  from \eqref{rhoEquation} we can conclude  that for any $\theta \in\,(0,1/2)$
\begin{equation}\label{uniformHolder}
    \sup_{\mu \in\,(0,\mu_T)}\mathbb{E} \,\Vert \rho_\mu +\mu \partial_t u_\mu \Vert_{C^\theta ([0,T];H^{-1})} <\infty.
\end{equation}
Moreover,  thanks to estimates \eqref{[2]} and \eqref{cor2} we have 
\begin{equation}\label{UniformBounded}
   \sup_{\mu \in\,(0,\mu_T)} \mathbb{E}\Vert \rho_\mu +\mu \partial_t u_\mu \Vert_{C([0,T];H)}<\infty.
\end{equation}

Now,  due to \eqref{uniformHolder} and \eqref{UniformBounded}, for any $\epsilon>0$ there exist two constants $L^\e_1,L^\e_2>0$ such that, if we define
$$K^\e_1=\{f:[0,T]\times \mathbb{R}\to \mathbb{R}\ : \Vert f\Vert _{C^\theta([0,T];H^{-1})} \leq L^\e_1\}$$
and 
$$K^\e_2=\{f:[0,T]\times \mathbb{R}\to \mathbb{R}\ : \Vert f\Vert _{C([0,T];H)} \leq L^\e_2\},$$
then 
\begin{equation}\label{tightInequality1}
\inf_{\mu \in\,(0,\mu_T)  }    \mathbb{P}(\rho_\mu+\mu \partial_t u_\mu\in K^\e_1\cap K^\e_2)>1-\frac{\epsilon}{3}.
\end{equation}
By the compact embedding of $H$ into $H^{-\d}$,  we know that $K^\e_1 \cap K^\e_2$ is relatively compact in $C([0,T];H^{-\d})$, for every $\d>0$ (for a proof see  \cite[Theorem 5]{Simon1986}). Therefore, $K_1^\e\cap K_2^\e$ is relatively compact in $X_1$.

In Proposition \ref{Energy estimate}, we have shown that 
\[\lim_{\mu\to 0}\,\mathbb{E}\Vert \mu \partial_tu_\mu \Vert_{C([0,T];H)}^2=0.\]
Hence for every sequence $(\mu_k)_{k \in\,\mathbb{N}}\subset (0,\mu_T)$ converging to zero there is a compact set $K_3^\e$ in $C([0,T];H)$ such that 
\begin{equation}
\label{tightInequality2}
\mathbb{P}(-\mu_k \partial_t u_{\mu_k}\in K_3^\e)>1-\frac{\e}{6},\ \ \ \ \ k \in\,\mathbb{N}.
\end{equation}
Since $C([0,T];H)\subset X_1$, $K^\e_3$ is  also compact in $X_1$.
Then $(K^\e_1 \cap K^\e_2)+K_3^\e$ is relatively compact in $X_1$, and thanks to  \eqref{tightInequality1} and \eqref{tightInequality2},
for every $k \in\,\mathbb{N}$
\begin{align}\label{tightInequality4}
    \mathbb{P}(\rho_{\mu_k} \in (K^\e_1 \cap K^\e_2)+K_3^\e)\geq \mathbb{P}(\rho_{\mu_k}+\mu_k \partial u_{\mu_k} \in K^\e_1\cap K^\e_2,-\mu_k \partial_t u_{\mu_k}\in K_3^\e)> 1-\frac{\epsilon}{2}.
\end{align}
By the arbitrariness of $\epsilon>0$, this means that the family of probability measures $(\mathcal{L}(\rho_{\mu_k}))_{k\in \mathbb{N}}$ is tight in $X_1$.

Now, due to the characterization given in \cite[Theorem 1]{Simon1986} for compact sets in $C([0,T];H^{-\d})$, if for every $h \in\,(0,T)$ we define 
\[\tau_h f(t)=f(t+h),\ \ \ \ \ t \in\,[-h,T-h],\]
we have
\begin{equation}\label{tightInequality3}
    \lim_{h\rightarrow 0} \sup_{f\in (K^\e_1 \cap K^\e_2)+K_3^\e} \Vert \tau_h f -f\Vert_{C([0,T-h];H^{-\d})}=0,\ \ \ \ \d>0.
\end{equation}

Next, due to \eqref{cor2}, 
there exists $L^\e_4>0$ such that if we define 
\begin{equation*}
    K^\e_4=\{f:[0,T]\times \mathbb{R}\to \mathbb{R}\ :\ \Vert f\Vert_{L^2(0,T;H^1)}\leq L^\e_4\},
\end{equation*}
then 
\begin{equation}\label{tightInequality5}
\inf_{\mu \in\,(0,\mu_T)}  \mathbb{P}(\rho_\mu \in K^\e_4)>1-\frac{\epsilon}{2}.
\end{equation}
Thus, if we take 
\[K^\e:=[(K^\e_1 \cap K^\e_2)+K_3^\e]\cap K^\e_4,\]
 from \eqref{tightInequality4} and \eqref{tightInequality5} we obtain 
\begin{equation}
\label{gs1002}
\inf_{k \in\,\mathbb{N}}    \mathbb{P}(\rho_{\mu_k}\in K^\e)>1-\epsilon.
\end{equation}

Now, let us fix $p \in\,(2,\infty)$ and let us define
\[\d_p=\frac 2{p-2},\ \ \ \ \ \a_p=\frac{p-2}{p}.\]
 It is immediate to check that
\[\Vert x\Vert_H\leq c_p \Vert x\Vert_{H^{-\d_p}}^{\a_p}\Vert x\Vert_{H^1}^{1-\a_p}.\]
Due to \eqref{tightInequality3}, we have
\[  \lim_{h\rightarrow 0} \sup_{f\in K^\e} \Vert \tau_h f -f\Vert_{C([0,T-h];H^{-\d_p})}=0.\]
Moreover, $K^\epsilon$ is bounded in $L^2(0,T;H^1)$. Then, since
\[\frac {\a_p}\infty +\frac {1-\a_p}2 =\frac 1p,\]
according to \cite[Theorem 7]{Simon1986} we have that $K^\e$ is relatively compact in $L^p(0,T;H)$. Due to the arbitrariness of $p<\infty$, we have that  $K^\e$ is relatively compact in $X_2$. By the arbitrariness of $\e>0$ and \eqref{gs1002}, this allows us to conclude that  the family of probability measures $(\mathcal{L}(\rho_{\mu_k}))_{k\in \mathbb{N}}$ is tight in $X_2$.

\end{proof}

\section{Uniqueness for the quasilinear parabolic equations}
\label{sec4}
In this section, we prove the uniqueness of solutions for the following quasilinear stochastic parabolic equation
\begin{equation}
\label{SPDERho}
\le\{\begin{array}{l}
\ds{\partial_t \rho=\text{div}[b(\rho) \nabla \rho]+F(\rho)+\si_g(\rho)dw^Q(t), \ \ \ \  t>0, \ \ \ \ x\in \mathcal{O};}\\[10pt]
\ds{\rho(0,x)=g(u_0),\ \ \ \ \ \ \ \ \rho(t,x)=0,\ \ \ x\in \partial \mathcal{O},}
\end{array}\r.
\end{equation}
where, we recall 
\[
b(r)=\frac 1{\gamma(g^{-1}(r))},\ \ \ \ F(r)=(f\circ g^{-1})(r),\ \ \ r \in\,\mathbb{R}, \]
and 
\[\si_g(h)=\si(g^{-1}\circ h),\ \ \ \ h \in\,H.\]
Notice that because of our assumptions on $\gamma$ and $f$, the functions $b$ and $F$ are both globally Lipschitz continuous on $\mathbb{R}$ and  the mapping $\si_g:H\to \mathcal{L}_2(H_Q,H)$ is bounded and Lipschitz continuous. 

\begin{Definition}
{\em An $(\mathcal{F}_t)_{t\geq 0}$ adapted process $\rho \in\, L^2(\Omega; C([0,T];H^{-1}))\cap L^2(\Omega;L^2(0,T;H^1))$ is said to be a solution of equation \eqref{SPDERho} if for every test function $\psi\in C^\infty_0(\mathcal{O})$
\begin{equation}
\label{gs68}
\begin{aligned}
    \langle \rho(t),\psi\rangle_H& =\langle g(u_0),\psi\rangle_H -\int_0^t \langle b(\rho(s))\nabla \rho(s),\nabla \psi\rangle_H ds\\[10pt]
    & \quad +\int_0^t \langle F(\rho(s)),\psi\rangle_H ds +\int_0^t\langle \si_g(\rho(s))dw^Q(s),\psi\rangle_H.
    \end{aligned}
\end{equation}}
\end{Definition}

\begin{Theorem}\label{RhoUniquenessTheorem}
Suppose Assumptions \ref{Assumption1}, \ref{Assumption2} and \ref{Assumption3} are satisfied. Then there is at most one solution $\rho \in\, L^2(\Omega; C([0,T];H^{-1}))\cap L^2(\Omega;L^2(0,T;H^1))$ to equation \eqref{SPDERho}.
\end{Theorem}

\begin{proof} The proof is a slight modification of  \cite[Proof of Theorem 3.1]{HZ2017}, where Hofmanov\'a and Zhang use a generalized It\^o formula for the $L^1$-norm of solutions of the same class of stochastic quasilinear parabolic equations.  In \cite{HZ2017} the periodic boundary condition on the torus $\mathbb{T}^n$ is considered and this means that the authors can take the identity function on the torus as a test function. Since we are considering here  Dirichlet boundary conditions,  we have to use a different class of test functions.

Let $(\varphi_n)_{n \in\,\mathbb{N}}$ be the sequence of functions constructed in  \cite[Proof of Theorem 3.1]{HZ2017}, which have bounded first and second order derivatives, 
\begin{equation}\label{quasiUnique1}
\varphi_n ' (0)=0, \qquad     \vert \varphi_n ' (r)\vert \leq 1, \qquad  0\leq \varphi_n ''(r) \leq \frac{2}{n\vert r\vert},\ \ \ \ r \in\,\mathbb{R},
\end{equation}
and 
\begin{equation}
\label{gs69}
\lim_{n\to \infty} \sup_{r \in\,\mathbb{R}}|\varphi_n(r) -|r||=0.
\end{equation}  

Now, suppose $\rho_1,\rho_2\in L^2(\Omega; C([0,T];H^{-1}))\cap L^2(\Omega;L^2(0,T;H^1))$ are both  solutions to \eqref{SPDERho}. By the generalized It\^o formula in Proposition \ref{GeneralIto}, for any test function $\psi \in C_0^\infty(\mathcal{O})$ we have 
\begin{equation}
\label{gx006}
\langle \varphi_n (\rho_1(t)-\rho_2(t)),\psi\rangle_H=:\sum_{k=1}^5 I_{k,n}(t),	
\end{equation}
where
\[I_{1,n}(t):=\int_0^t \langle \varphi_n^\prime(\rho_1(s)-\rho_2(s))(F(\rho_1(s))-F(\rho_2(s))),\psi\rangle_H ds ,\]
\[I_{2,n}(t):=- \int_0^t \langle \varphi_n^{\prime \prime}(\rho_1(s)-\rho_2(s))(\nabla \rho_1(s)-\nabla \rho_2(s))\cdot (b(\rho_1(s))\nabla \rho_1(s)-b(\rho_2(s))\nabla \rho_2(s)),\psi\rangle_H ds,\]

\[I_{3,n}(t):=-\int_0^t \langle \varphi_n '(\rho_1(s)-\rho_2(s))(b(\rho_1(s))\nabla \rho_1(s)-b(\rho_2(s))\nabla \rho_2(s)),\nabla\psi\rangle_H ds,\]
\[I_{4,n}(t):=\frac 12\int_0^t\langle \varphi_n^{\prime \prime}(\rho_1(s)-\rho_2(s))\sum_{i=1}^\infty \le\vert\le[\si_g(\rho_1(s))-\si_g(\rho_2(s))\r]Qe_i\r\vert^2,\psi\rangle_H\,ds,\]
and
\[I_{5,n}(t):=\int_0^t\langle \varphi_n^\prime(\rho_1(s)-\rho_2(s))\le[\si_g(\rho_1(s))-\si_g(\rho_2(s))\r]dw^Q(s),\psi\rangle_H.\]

By the boundedness of $\varphi_n'$ and $\varphi_n''$, \eqref{gx006} is also valid for any $\psi\in H^1 \cap C(\mathcal{O})$ with $\psi=0$ on $\partial \mathcal{O}$ by approximation, i.e. there exist $\psi_n\in C_0^\infty(\mathcal{O})$ converging to $\psi$ in both $L^\infty$ and $H^1$ norms. In particular, here we take the test function $\psi$ to be positive superharmonic with non-positive $\Delta \psi \in L^2$. 
Thanks to \eqref{quasiUnique1} and the Lipschitz continuity of $F$,
\begin{equation*}
    I_{1,n}(t)\leq c\int_0^t \langle\vert \rho_1(s)-\rho_2(s)\vert, \psi \rangle_H ds.
\end{equation*}
For the second term, thanks to \eqref{eq:AUniform},  \eqref{quasiUnique1} and the Lipschitz continuity of $b$
\begin{align*}
    I_{2,n}(t)&=-\int_0^t \langle \varphi_n ''(\rho_1(s)-\rho_2(s)) b(\rho_1(s))(\nabla \rho_1(s)-\nabla \rho_2(s))\cdot (\nabla \rho_1(s)-\nabla \rho_2(s)),\psi\rangle_H ds\\[10pt]
    &\quad-\int_0^t \langle \varphi_n ''(\rho_1(s)-\rho_2(s)) (b(\rho_1(s))-b(\rho_2(s))) (\nabla \rho_1(s)-\nabla \rho_2(s))\cdot\nabla \rho_2(s),\psi\rangle_H ds\\[10pt]
    &\leq \frac{c}{n}\int_0^t \langle \vert\nabla \rho_1(s)-\nabla \rho_2(s)\vert \vert\nabla \rho_2(s)\vert, \psi \rangle_H ds\\[10pt]
    &\leq \frac{c\Vert \psi\Vert_{L^\infty(\mathcal{O})}}{n}\int_0^t \le(\Vert \rho_1(s)\Vert_{H^1}^2+ \Vert \rho_2(s)\Vert_{H^1}^2 \r)ds.
\end{align*}
For the third term, by the definition of $b$, we have $b(\rho)\nabla \rho=\nabla g^{-1}(\rho)$, from which we have
\begin{align*}
 \ds{    I_{3,n}(t) }&= \ds{-\int_0^t \langle \varphi_n '(\rho_1(s)-\rho_2(s))(\nabla g^{-1}(\rho_1(s))-\nabla g^{-1}(\rho_2(s))),\nabla\psi\rangle_H ds}\\[10pt]
     &=\int_0^t \langle \varphi_n ''(\rho_1(s)-\rho_2(s))(\nabla \rho_1(s)-\nabla \rho_2(s))( g^{-1}(\rho_1(s))- g^{-1}(\rho_2(s))),\nabla\psi\rangle_H ds\\[8pt]
     &\quad +\int_0^t \langle \varphi_n '(\rho_1(s)-\rho_2(s))( g^{-1}(\rho_1(s))- g^{-1}(\rho_2(s))),\Delta\psi\rangle_H ds.
\end{align*}
Thanks to \eqref{quasiUnique1} and the Lipschitz continuity of $g^{-1}$, 
\begin{equation*}
    \vert \varphi_n ''(\rho_1(s)-\rho_2(s))( g^{-1}(\rho_1(s))- g^{-1}(\rho_2(s)))\vert\leq c\vert \varphi_n ''(\rho_1(s)-\rho_2(s))\vert \,\vert \rho_1(s)- \rho_2(s)\vert\leq \frac{c}{n}.
\end{equation*}
Since   $\varphi^\prime_n$ is increasing and $\varphi^\prime_n(0)=0$, we have $\text{sign} \,\varphi^\prime_n(r)=\text{sign}\, r$. Then, as  $g^{-1}$ is also increasing, we have
\[\varphi_n '(r_1-r_2)( g^{-1}(r_1)- g^{-1}(r_2))\geq 0,\ \ \ \ \ \ r_1, r_2 \in\,\mathbb{R}.\]
Together with $\Delta\psi\leq 0$ on $\mathcal{O}$, for every $t \in\,[0,T]$ we have 
\begin{equation*}
    I_{3,n}(t)\leq \frac{c}{n}\int_0^t \langle \vert \nabla \rho_1(s)-\nabla \rho_2(s)\vert, \vert \nabla \psi\vert\rangle_H ds\leq \frac{c_T\Vert  \psi\Vert_{H^1}}{n}\int_0^t \le(\Vert \rho_1(s)\Vert_{H^1}^2+ \Vert \rho_2(s)\Vert_{H^1}^2 \r)ds.
\end{equation*}
For the fourth term, by Assumption \ref{Assumption1} we have
\begin{equation*}
    \begin{aligned}
     I_{4,n}(t)&= \frac 12\int_0^t\langle \varphi_n^{\prime \prime}(\rho_1(s)-\rho_2(s))\sum_{i=1}^\infty \le\vert\sigma_i(\cdot,g^{-1}(\rho_1(s)))-\sigma_i(\cdot,g^{-1}(\rho_2(s)))\r\vert^2,\psi\rangle_H\,ds\\[10pt]
     &\leq \frac c2\int_0^t\langle \varphi_n^{\prime \prime}(\rho_1(s)-\rho_2(s))\le\vert\rho_1(s)-\rho_2(s)\r\vert^2,\psi\rangle_H\,ds\\[10pt]
     &\leq \frac{c\Vert \psi\Vert_H}{n}\int_0^t (\Vert \rho_1(s)\Vert_H+\Vert \rho_2(s)\Vert_H) ds\\[10pt]
     &\leq \frac{c_T\Vert \psi\Vert_{H^1}}{n}\int_0^t (\Vert \rho_1(s)\Vert_{H^1}^2+\Vert \rho_2(s)\Vert_{H^1}^2) ds.
    \end{aligned}
\end{equation*}
Therefore, we take the expectation of \eqref{gx006} and combine the estimates for $I_{1,n}(t)$, $I_{2,n}(t)$, $I_{3,n}(t)$, and $I_{4,n}(t)$ to obtain
\begin{align*}
    \mathbb{E}  \langle \varphi_n (\rho_1(t)-\rho_2(t)),\psi\rangle_H & \leq \frac{c_T}{n}\le(\Vert  \psi\Vert_{H^1}+\Vert \psi\Vert_{L^\infty(\mathcal{O})}\r)\int_0^t \le(\mathbb{E}\Vert \rho_1(s)\Vert_{H^1}^2+ \mathbb{E}\Vert \rho_2(s)\Vert_{H^1}^2 \r)ds\\[10pt]
    &\quad +c\,\int_0^t \mathbb{E}\langle\vert \rho_1(s)-\rho_2(s)\vert,\psi\rangle_H ds.
\end{align*}
Now, we take the limit above, as $n\rightarrow \infty$,  and we get 
\begin{align*}
 \mathbb{E}   \langle \vert\rho_1(t)-\rho_2(t)\vert ,\psi\rangle_H \leq c \int_0^t \mathbb{E}\langle\vert \rho_1(s)-\rho_2(s)\vert,\psi\rangle_H ds,
\end{align*}
which implies that 
\[\langle \vert\rho_1(t)-\rho_2(t)\vert ,\psi\rangle_H=0, \ \ \ \ \ \  \text{a.s. on }\Omega\times [0,T].\]
 Since this is true for all positive superharmonic $\psi\in C(\mathcal{O}) \cap H^1$ with zero boundary value and non-positive $\Delta \psi \in L^2$, we have $\rho_1=\rho_2$ and the uniqueness follows.
 \end{proof}

\section{The convergence result}
\label{sec7}

Now we are ready to prove the convergence of the solutions to \eqref{SPDE1} and identify the limit as the unique solution of the quasilinear parabolic equation \eqref{gs40}.

\begin{Theorem}
Suppose Assumptions \ref{Assumption1}, \ref{Assumption2} and \ref{Assumption3} are satisfied and, for each $\mu>0$, let $(u_\mu,\partial_t u_\mu)$ denote the unique solution to equation \eqref{SPDE1} with the same initial condition $(u_0,v_0) \in\,\mathcal{H}_1$. Then for every $\d>0$ and $p<\infty$, and for every $\eta>0$
\[\lim_{\mu \to 0} \mathbb{P}\le(\Vert u_\mu-u\Vert_{C([0,T];H^{-\d})}+\Vert u_\mu-u\Vert_{L^p(0,T;H)}>\eta \r)=0,\] where 
$u \in\,L^2(\Omega; X_1\cap X_2\cap L^2(0,T;H^1))$ is the unique solution of equation \eqref{gs40}, with initial datum $u_0$.
\end{Theorem}

\begin{Remark}
{\em 
Here we only consider  deterministic initial data $(u_0,v_0) \in\,H^1(\mathcal{O})\times L^2( \mathcal{O})$, independent of $\mu$. Actually, it is easy to generalize our result to the cases of random initial data $(u_\mu(0),\partial_t u_\mu(0)) \in\,H^1(\mathcal{O})\times L^2(\mathcal{O})$, depending on $\mu$, such that for some $u_0 \in\,H^1(\mathcal{O})$
\[\lim_{\mu\to 0}\le(\mathbb{E}\Vert u_\mu(0)-u_0\Vert^2_{H^1(\mathcal{O})}+\mu^2\,\mathbb{E}\Vert\partial_tu_\mu(0)\Vert^2_ {L^2(\mathcal{O})}\r)=0.\]}
\end{Remark}

\begin{proof} We recall that in the previous section we have introduced the two Polish spaces $X_1$ and $X_2$, endowed with the distances $d_1$ and $d_2$, defined in \eqref{d1} and \eqref{d2}, respectively.
Here, for every $T>0$ we denote
\[\mathcal{K}_T:=[X_1\cap X_2]^2\times [C([0,T];H)]^2\times  C([0,T];U),\]
where $U$ is the Hilbert space containing $H_Q$, with Hilbert-Schmidt embedding (see \eqref{contb}).

In Theorem \ref{theorem:tightness} we have proved that for any sequence $(\mu_k)_{k \in\,\mathbb{N}}$ converging to zero, the
sequence $(\L(\rho_{\mu_k},\mu_{k}\, \partial_t u_{\mu_k}))_{k\in \, \mathbb{N}}$ is tight in $[X_1\cap X_2]\times C([0,T];H)$. Hence, the Skorokhod theorem
assures that, for any two sequences $(\mu^1_k)_{k \in\,\mathbb{N}}$ and $(\mu^2_k)_{k \in\,\mathbb{N}}$
converging to zero, there exist  two subsequences, still denoted by $(\mu^1_k)_{k \in\,\mathbb{N}}$ and $(\mu^2_k)_{k \in\,\mathbb{N}}$, a sequence of
random variables
\[Y_k:=\le((\rho^1_k,\vartheta^1_k),(\rho^2_k,\vartheta^2_k),\hat{w}_k^Q\r),\ \ \ \ k \in\,\nat,\]
in $\mathcal{K}_T$, and a random variable 
\[Y:=(\rho^1,\rho^2,\hat{w}^Q),\]
in $[X_1\cap X_2]^2\times  C([0,T];U)$,
all defined on some
probability space $(\hat{\Omega},\hat{\F},\hat{\Pro})$, such that
\begin{equation}
\label{gs65}
\mathcal{L}(Y_k)=\mathcal{L}\le((\rho_{\mu^1_{k}},\mu^1_{k}\,\partial_t u_{\mu^1_{k}}),(\rho_{\mu^2_{k}},\mu^2_{k}\,\partial_t u_{\mu^2_{k}}),w^Q\r),\ \ \ \ k \in\,\nat,\end{equation}
 and for $i=1, 2$
\begin{equation}
\label{gs66}
    \lim_{k\to\infty}\le(\Vert \rho^i_{k}-\rho^i\Vert_{X_1}+\Vert \rho^i_{k}-\rho^i\Vert_{X_2}+\Vert \vartheta^i_k\Vert_{C([0,T];H)}+\Vert \hat{w}_k^Q-\hat{w}^Q\Vert_{C([0,T];U)}\r)=0,\ \ \ \ \hat{\mathbb{P}}-\text{a.s.}.
\end{equation}
 Notice that, due to \eqref{gs24quater} and \eqref{gs66}, we have
\begin{equation}
\label{gs70}
\rho^i \in\,L^2(\Omega;X_1\cap X_2\cap L^2(0,T;H^1)),\ \ \ \ i=1, 2.
\end{equation}
Next, a filtration $(\hat{\mathcal{F}}_t)_{t\geq 0}$ is introduced in $(\hat{\Omega},\hat{\mathcal{F}},\hat{\mathbb{P}})$, by taking the augmentation of the canonical filtration of $(\rho^1, \rho^2,\hat{w}^Q)$, generated by the restrictions of  $(\rho^1, \rho^2,\hat{w}^Q)$ to every interval $[0,t]$. Due to this construction, $\hat{w}^Q$ is a $(\hat{\mathcal{F}}_t)_{t\geq 0}$ Wiener process with covariance $Q^*Q $ (for a proof see \cite[Lemma 4.8]{DHV2016}).

Now, if we show that $\rho^1=\rho^2$, we have that $\rho_\mu$ converges in probability to some $\rho
\in\,L^2(\Omega;X_1\cap X_2\cap L^2(0,T;H^1))$. Actually, as observed by Gy\"ongy and Krylov in
\cite{gk}, if $E$ is any Polish space equipped with the Borel
$\si$-algebra, a sequence $(\xi_n)_{n \in\,\mathbb{N}}$ of $E$-valued random
variables converges in probability if and only if for every pair
of subsequences $(\xi_m)_{m \in\,\mathbb{N}}$ and $(\xi_l)_{l \in\,\mathbb{N}}$ there exists an
$E^2$-valued subsequence $\eta_k:=(\xi_{m(k)},\xi_{l(k)})$
converging weakly to a random variable $\eta$ supported on the
diagonal $\{(h,k) \in\,E^2\ :\ h=k\}$.

In order to show that $\rho_1=\rho_2$, we prove that they are both a solution of equation \eqref{SPDERho}, which has pathwise uniqueness due to Theorem \ref{RhoUniquenessTheorem}. To this purpose, we use the general method introduced in \cite{DHV2016}.

Due to \eqref{gs65}, both $(\rho^1_k,\vartheta_k^1)$ and $(\rho^2_k,\vartheta_k^2)$ satisfy equation \eqref{rhoEquation}, with $w^Q$ replaced by $\hat{w}_k^Q$. Then, by first taking the scalar product in $H$ of each term in  \eqref{rhoEquation} with an arbitrary but fixed  $\psi\in C^\infty_0(\mathcal{O})$ and then integrating by parts, we get
\begin{equation}\label{rhoApproxi}
\begin{aligned}
\langle \rho^i_k(t)+\vartheta^i_k(t),\psi\rangle_H &=\langle g(u_0)+\mu_k v_0,\psi\rangle_H -\int_0^t \langle b(\rho^i_k(s)) \nabla \rho^i_k(s),\nabla \psi \rangle_H ds\\[10pt]
&\quad+\int_0^t \langle F(\rho^i_k(s)),\psi \rangle_H ds +\int_0^t\langle \si_g(\rho^i_k(s))d\hat{w}_k^Q (s),\psi\rangle_H,\ \ \ \ \ i=1, 2.
\end{aligned}    
\end{equation}

Now, since $g$ is invertible, we can define
\begin{equation*}
    u^i_k(t,x)=g^{-1}(\rho^i_k(t,x)),\ \ \ \ \ u^i(t,x)=g^{-1}(\rho^i(t,x)),\ \ \ \ (t,x) \in\,[0,T]\times \mathcal{O}.\end{equation*}
Due to the Lipschitz continuity of $g^{-1}$, we have that  $u^i_k$ and $u^i$ belong to $L^2(\Omega;X_1\cap X_2\cap L^2(0,T;H^1))$ and, in view of \eqref{gs66}
\begin{equation}
\label{gs58}
\lim_{k\to\infty} \le(\Vert u^i_{k}-u^i\Vert_{X_1}+\Vert u^i_{k}-u^i\Vert_{X_2}\r)=0,\ \ \ \ \ \hat{\mathbb{P}}-\text{a.s}.\end{equation}
Moreover 
\begin{equation*}
    \nabla u^i_k(s)=b(\rho^i_k(s)) \nabla \rho^i_k(s),\qquad \nabla u^i(s)=b(\rho^i(s)) \nabla \rho^i(s),
\end{equation*}
so that 
\begin{align*}
    &\int_0^t \langle b(\rho^i_k(s)) \nabla \rho^i_k(s),\nabla \psi \rangle_H ds -\int_0^t \langle b(\rho^i(s)) \nabla \rho^i(s),\nabla \psi \rangle_H ds\\[10pt]
    &=\int_0^t \langle  \nabla u^i_k(s),\nabla \psi \rangle_H ds -\int_0^t \langle \nabla u^i(s),\nabla \psi \rangle_H ds\\[10pt]
    &=-\int_0^t \langle  (u^i_k(s)-u^i(s)),\Delta \psi \rangle_H ds.
\end{align*}
In particular, due to \eqref{gs58}, we have that 
\begin{equation}\label{eq:gs68}
    \lim_{k\to \infty} \int_0^t \langle b(\rho^i_k(s)) \nabla \rho^i_k(s),\nabla \psi \rangle_H ds =\int_0^t \langle b(\rho^i(s)) \nabla \rho^i(s),\nabla \psi \rangle_H ds,\ \ \ \ \hat{\mathbb{P}}-\text{a.s.}
\end{equation}
Now, for $i=1, 2$ and $t \in\,[0,T]$, we define
\[\hat{M^i}(t)=\langle \rho^i(t),\psi\rangle_H-\langle g(u_0),\psi\rangle_H +\int_0^t \langle b(\rho^i(s)) \nabla \rho^i_k(s),\nabla \psi \rangle_H ds-\int_0^t \langle F(\rho^i(s)),\psi \rangle_H ds.\]
By proceeding as in the proof of \cite[Lemma 4.9]{DHV2016}, thanks to \eqref{gs66}, \eqref{eq:gs68} and the Lipschitz continuity of $F$,  we have that for every $t \in\,[0,T]$
\[\le<\hat{M^i}-\int_0^\cdot \langle \sigma_g (\rho^i(s))d\hat{w}^Q(s),\psi\rangle_H\r>_t=0,\ \ \ \ \mathbb{P}-\text{a.s},\]
where $\langle\cdot\rangle_t$ is the quadratic variation process. This implies that
both $\rho^1$ and $\rho^2$ satisfy equation \eqref{SPDERho}. Namely, for every $\psi\in\,C^\infty_0(\mathcal{O})$ and $i=1, 2$
\begin{align*}
    \langle \rho^i(t),\psi\rangle_H=&\langle g(u_0),\psi\rangle_H -\int_0^t \langle b(\rho^i(s)) \nabla \rho^i(s),\nabla \psi \rangle_H ds \\[10pt]
    & +\int_0^t \langle F(\rho^i(s)),\psi \rangle_H ds +\int_0^t\langle \si_g(\rho^i(s))d\hat{w}^Q (s),\psi\rangle_H.
\end{align*}

As we have recalled above, thanks to the remark by Gy\"ongy-Krylov  in
\cite{gk} this implies that $\rho_\mu$ converges in probability to some random variable  $\rho $ taking values in $X_1\cap X_2$, as $\mu$ goes to zero. Due to \eqref{gs70}, we also have that $\rho$ belongs to $L^2(\Omega;X_1\cap X_2\cap L^2(0,T;H^1))$ and satisfies equation \eqref{SPDERho}.

Now we set 
\[u=g^{-1}(\rho).\]
 Due to the Lipschitz continuity of $g^{-1}$ we have that $u \in\,L^2(\Omega;X_1\cap X_2\cap L^2(0,T;H^1))$ and $u_\mu$ converges in probability to $u$ in $X_1\cap X_2$, as $\mu$ goes to zero.
In order to conclude, we have to identify $u$ with the solution of equation \eqref{gs40}. We apply the generalized It\^o formula stated in Proposition \ref{ItoFormulaSPDE} to $u:=g^{-1}(\rho)$ with 
\[\mathfrak{U}=H_Q,\ \ \ \ J_i(t)=\si_g(\rho(t))Q e_i,\quad i\in\mathbb{N},\]
and 
\[F(t)=F(\rho(t)),\quad G(t)=b(\rho(t))\nabla \rho(t)).\]
 Actually, since 
\begin{equation*}
    (g^{-1})'(r)=\frac{1}{\gamma(g^{-1}(r))},\ \ \ \ \  (g^{-1})''(r)= -\frac{\gamma '(g^{-1}(r))}{\gamma(g^{-1}(r))^3},\ \ \ \ r \in\,\mathbb{R},
\end{equation*}
for any $\psi\in C^\infty_0(\mathcal{O})$ we can conclude that
\begin{align*}
\langle u(t),\psi\rangle_H &=\langle u_0,\psi\rangle_H -\int_0^t \le\langle \frac{\nabla u(s)}{\gamma(u(s))}, \nabla\psi \r\rangle_H ds-\int_0^t \le\langle \nabla \le(\frac{1}{\gamma(u(s))} \r)\cdot \nabla u(s),\psi \r\rangle_H  ds\\[10pt]
&\quad +\int_0^t \le\langle \frac{f(u(s))}{\gamma(u(s))},\psi \r\rangle_H ds -\int_0^t \le\langle \frac{\gamma '(u(s))}{2\gamma(u(s))^3} \,\sum_{i=1}^\infty (\si(u(s))Qe_i)^2,\psi \r\rangle_H ds\\[10pt] 
&\quad+\int_0^t \le\langle \frac{\si(u(s))}{\gamma(u(s))} \,d w^Q (s),\psi \r\rangle_H,
\end{align*}
which means that $u$ is a solution to \eqref{gs40}. 

In order to prove the uniqueness of the solution of equation \eqref{gs40}, if $u_1$ and $u_2$ are two solutions, we apply
 Proposition \ref{ItoFormulaSPDE} to $\rho_j=g(u_j)$, $j=1, 2$, with 
 \[\mathfrak{U}=H_Q,\ \ \ \ J^j_i(t)=\frac{\si(u_j)}{\gamma(u_j)}\,Q e_i, \ \ \ G^j(t)=\frac{\nabla u_j(t)}{\gamma(u_j(t))}, \quad i\in\mathbb{N},\]
 and 
 \[F^j(t)=\frac{f(u_j(t))}{\gamma(u_j(t))}-\gamma(u_j(t)) \nabla\le(\frac{1}{\gamma(u_j(t))}\r)\cdot \nabla u_j(t) -\frac{\gamma '(u_j(t))}{2\gamma(u_j(t))^3} \,\sum_{i=1}^\infty (\si(u_j(t))Qe_i)^2.\]
Then it turns out that both $g(u_1)$ and $g(u_2)$ are solutions to \eqref{SPDERho}. Thus, by the uniqueness result in Theorem \ref{RhoUniquenessTheorem}, we can conclude that $g(u_1)=g(u_2)$, and this implies that  $u_1=u_2$. 
\end{proof}

\appendix

\section{A generalized It\^o formula}\label{gxappend}

In \cite{DHV2016},  it proved a generalized It\^o formula for the weak solutions of the following general class of equations 
\begin{equation}\label{ItoFormulaSPDE}
    du(t)=F(t)dt+\text{div}\, G(t) dt + J(t)\,dw(t),\ \ \ \ u_0 \in\,H,
\end{equation}
where $H=L^2(\mathbb{T}^d)$, $d\geq 1$. In the present paper we are dealing with Dirichlet boundary conditions in general bounded open sets $\mathcal{O}$. In what follows we adapt the formulation of \cite[Proposition A.1]{DHV2016} to our situation and we briefly describe the  modification we have to do in the proof. 

\begin{Proposition}\label{GeneralIto}
Let $\psi\in C^\infty_0(\mathcal{O})$ and $\varphi\in C^2(\mathbb{R})$, with bounded second-order derivative. Suppose $W$ is a space-time white noise, that is
\[w(t)=\sum_{i=1}^\infty e_i\beta_i(t),\]
 where $(\beta_i)_{i\in\mathbb{N}}$ are mutually independent standard Wiener processes on the stochastic basis $(\Omega,\mathcal{F},(\mathcal{F}_t)_{t\in[0,T]},\mathbb{P})$ and $(e_i)_{i\in\mathbb{N}}$ is a complete orthonormal system in a separable Hilbert space $\mathfrak{U}$. Assume that $F$ and $G_j$ are adapted processes in   $L^2(\Omega;L^2(0,T;H))$, $j=1,\ldots,d$ and $J$ is an adapted process in  $L^2(\Omega;L^\infty(0,T;L_2(\mathfrak{U};H)))$. For every $i \in\,\mathbb{N}$, let $J_i(t):=J(t)e_i$. If the process 
\begin{equation*}
    u\in L^2(\Omega; C([0,T];H^{-1}))\cap L^2(\Omega;L^2(0,T;H^1))
\end{equation*}
solves \eqref{ItoFormulaSPDE} in $H^{-1}$, then almost surely, for all $t\in[0,T]$, 
\begin{equation}\label{ItoLimit}
\begin{aligned}
\langle \varphi(u(t)),\psi\rangle_H &=\langle \varphi(u_0),\psi\rangle_H + \int_0^t \langle \varphi '(u(s))F(s),\psi\rangle_H ds-\int_0^t \langle \varphi ''(u(s))\nabla u(s)\cdot G(s),\psi\rangle_H ds\\[10pt]
&\quad - \int_0^t \langle \varphi '(u(s))G(s),\nabla \psi\rangle_H ds+\frac{1}{2} \int_0^t \langle \varphi ''(u(s))\sum_{i=1}^\infty J_i^2(s),\psi\rangle_H\\[10pt]
&\quad +\int_0^t \langle \varphi '(u(s))J(s)\,dw(s),\psi\rangle_Hds.
\end{aligned}
\end{equation}
Moreover, if we further assume that $\varphi$ has bounded first-order derivative, the assumption on $F$ could be relaxed to $L^1(\Omega;L^1(0,T;L^1(\mathcal{O})))$ and we still have \eqref{ItoLimit} to be true. 
\end{Proposition}
\begin{proof}
It is enough to prove the result for any smooth $\psi$ with compact support in $\mathcal{O}$. Given a fixed $\psi\in C_0^\infty(\mathcal{O})$, suppose it is supported on the compact set $K\subset \mathcal{O}$ and let $\delta_0:=d(K,\mathcal{O}^c)>0$. We fix a positive smooth function $\xi$   supported on the unit ball with integral equals to $1$, and define   $\xi_\delta(x)=\frac{1}{\delta^d}\xi(\frac{x}{\delta})$. 
Then, if for any $f\in H$ we define $f^\delta=f\ast \xi_\delta$, for $\d<\d_0$, we have
\begin{equation*}
    \Vert f^\delta\Vert_{L^2(K)}\leq \Vert f\Vert_{H},\qquad \Vert f^\delta -f\Vert_{L^2(K)}\rightarrow 0.
\end{equation*}
Now, we apply the mollifiers $\xi_\d$  to $u(t)$ and we have
\begin{equation*}
    u^\delta(t,x)=u_0^\delta(x)+\int_0^t F^\delta(s,x) ds+\int_0^t \text{div}\, G^\delta(s,x) ds +\sum_{i=1}^\infty \int_0^t J_i^\delta(s,x) d\beta_i(s),
\end{equation*}
for all $x\in K$. Thus,  we can apply the It\^o formula to $\varphi(u^\delta(t,x))\psi(x)$ and, after we   integrate in $x$, we get
\begin{equation}\label{ItoApprox}
\begin{aligned}
 \langle\varphi(u^\delta(t)),\psi\rangle_H &= \langle\varphi(u^\delta_0),\psi\rangle_H + \int_0^t \langle \varphi '(u^\delta(s))F^\delta(s),\psi\rangle_H ds+\int_0^t\langle \varphi '(u^\delta(s))\text{div}\, G^\delta(s),\psi\rangle_H ds\\[10pt]
&\quad + \frac{1}{2}\sum_{i=1}^\infty\int_0^t \langle \varphi ''(u^\delta(s))(J_i^\delta(s))^2,\psi\rangle_H ds + \sum_{i=1}^\infty \int_0^t \langle \varphi'(u^\delta(s))J_i^\delta(s),\psi\rangle_H d\beta_i(s)\\[10pt]
&=\langle\varphi(u^\delta_0),\psi\rangle_H + \int_0^t \langle \varphi '(u^\delta(s))F^\delta(s),\psi\rangle_H ds+\int_0^t\langle \text{div}(\varphi '(u^\delta(s)) G^\delta(s)),\psi\rangle_H ds\\[10pt]
&\quad -\int_0^t\langle \varphi ''(u^\delta(s))\nabla u^\delta(s)\cdot G^\delta(s),\psi\rangle_H ds+ \frac{1}{2}\sum_{i=1}^\infty\int_0^t \langle \varphi ''(u^\delta(s))(J_i^\delta(s))^2,\psi\rangle_H ds \\[10pt]
&\quad+ \sum_{i=1}^\infty \int_0^t \langle \varphi'(u^\delta(s))J_i^\delta(s),\psi\rangle_H d\beta_i(s).
\end{aligned}
\end{equation}
At this point, using the same argument as in \cite{DHV2016}, we can take the limit above, as $\d$ goes to zero, and obtain \eqref{ItoLimit}.

\end{proof}

{\bf Acknowledgments.} The first named author would like to thank Viorel Barbu, Zdzis\l aw Brze\'zniak, Martina Hofmanov\`a and Irena Lasiecka for several interesting and useful conversations about this problem.

\end{document}